\documentclass[a4paper,11pt]{amsart}
\usepackage{amssymb,amsmath,amsfonts,amsthm}
\usepackage{latexsym}
\usepackage{amsbsy}
\usepackage{amssymb}
\usepackage{amsfonts}
\usepackage{amsmath}
\usepackage{tikz}
\usetikzlibrary{matrix,arrows}
\allowdisplaybreaks
 \newtheorem{thm}{Theorem}[section]
 
 \newtheorem{lem}[thm]{Lemma}
 \newtheorem{prop}[thm]{Proposition}
 \newtheorem{cor}[thm]{Corollary}
 \newtheorem{rem}[thm]{Remark}

 \newcommand{\bthm}{\begin{thm}}
 \newcommand{\ethm}{\end{thm}}
 \newcommand{\bd}{\begin{defin}}
 \newcommand{\ed}{\end{defin}}
 \newcommand{\blem}{\begin{lem}}
 \newcommand{\elem}{\end{lem}}
 \newcommand{\bcor}{\begin{cor}}
 \newcommand{\ecor}{\end{cor}}
 \newcommand{\bprop}{\begin{prop}}
 \newcommand{\eprop}{\end{prop}}
 \newcommand{\brem}{\begin{rem} \rm}
 \newcommand{\erem}{\end{rem}}
 \newcommand{\bex}{\begin{ex} \rm}
 \newcommand{\eex}{\end{ex}}

 \newcommand{\beq}{\begin{equation}}
 \newcommand{\eeq}{\end{equation} }

 \newcommand{\bea}{\begin{eqnarray}}
 \newcommand{\eea}{\end{eqnarray}}

 \newcommand{\beas}{\begin{eqnarray*}}
 \newcommand{\eeas}{\end{eqnarray*}}

 \newcommand{\beqs}{\begin{equation*}}
 \newcommand{\eeqs}{\end{equation*}}

 \newcommand{\bi}{\begin{itemize}}
 \newcommand{\ei}{\end{itemize}}

 \newcommand{\ben}{\begin{enumerate}}
 \newcommand{\een}{\end{enumerate}}

 \newcommand{\ba}{\begin{array}}
 \newcommand{\ea}{\end{array}}

 \newcommand{\ds}{\displaystyle}

 \newcommand{\R}{\mathbb R}

\newcommand{\NN}{\mathbb N}
\newcommand{\CC}{\mathbb C}
\newcommand{\RR}{\mathbb R}
\newcommand{\ZZ}{\mathbb Z}

\newcommand{\SSS}{\mathcal S}
\newcommand{\sss}{\mathfrak s}

\begin{document}

\title[The Weyl Pseudo-differential Operators with Radial Symbols]{$G$-type Spaces of Ultradistributions over $\mathbb{R}^d_+$ and the Weyl Pseudo-differential Operators with Radial Symbols}

\thanks{The paper was
supported by the project {\it Modelling and harmonic analysis
methods and PDEs with singularities}, No. 174024 financed by the
Ministry of Science, Republic of Serbia.}


\author[S. Jak\v si\'c]{Smiljana Jak\v si\'c}
\address{Smiljana Jak\v si\'c, Faculty of Forestry, University of Belgrade, Kneza Vi\v seslava 1, Belgrade, Serbia}
              \email{smiljana.jaksic@gmail.com}

\author[S. Pilipovi\'v]{Stevan Pilipovi\'c}
\address{Stevan Pilipovi\'c, Department of Mathematics, Faculty of Sciences, University of Novi Sad, Trg Dositeja Obradovi\'ca 4, Novi Sad, Serbia}
              \email{stevan.pilipovic@dmi.uns.ac.rs}

\author[B. Prangoski]{Bojan Prangoski}
\address{Bojan Prangoski, Faculty of Mechanical Engineering, University Ss. Cyril and Methodius, Karpos II bb, Skopje,
           Macedonia}
           \email{bprangoski@yahoo.com}


\begin{abstract}
The first part of the paper is devoted to the $G$-type spaces i.e.
the spaces $G^\alpha_\alpha (\mathbb R^d_+)$, $\alpha\geq 1$ and
their duals which can be described as analogous to the
Gelfand-Shilov spaces and their duals but with completely new
justification of obtained results. The Laguerre type expansions of
the elements in $G^\alpha_\alpha(\mathbb R^d_+)$, $\alpha\geq 1$
and their duals characterise these spaces through the exponential
and sub-exponential growth of coefficients. We provide the full
topological description and by the nuclearity of
$G_\alpha^\alpha(\mathbb{R}^d_+)$, $\alpha\geq 1$ the kernel
theorem is proved. The second part is devoted to the class of the
Weyl operators with radial symbols belonging to the $G$-type
spaces. The continuity properties of this class of
pseudo-differential operators  over the Gelfand-Shilov type spaces
and their duals are proved. In this way the class of the Weyl
pseudo-differential operators is extended to the one with the
radial symbols with the exponential and sub-exponential growth
rate.

%

\end{abstract}

\maketitle
 \section{Introduction}

The aim of this paper is twofold. In the first part, we study the
spaces $G^\alpha_\alpha(\R^d_+)$, $\alpha\geq 1$ and their strong
duals, analogous to the Gelfand-Shilov spaces and their duals,
while in the second part, we study a class of the Weyl operators
with rotationally invariant symbols using the results from the
first part.

The Gelfand-Shilov spaces $S_\alpha(\mathbb{R}^d)$,
$S^\beta(\mathbb{R}^d)$ and $S_\alpha^\beta(\mathbb{R}^d)$,
$\alpha+\beta\geq 1$ (referred to as $S$-type spaces) are very
well known and used in the analysis of Cauchy problems, spectral
analysis, time-frequency analysis and many other fields of
mathematics (see \cite{Kaminski}, \cite{GS1}-\cite{GPR2},
\cite{NR}). The spaces $S_\alpha^\alpha(\mathbb{R}^d)$,
$\alpha\geq 1/2$ are of the special interest because they are
invariant under Fourier transform and have been characterised
through the Hermite expansions and the corresponding estimates of
coefficients (see \cite{A}, \cite{M}, \cite{SP}).

On the other hand, the corresponding function spaces defined on
$\mathbb{R}_+^d$ are less studied, although they should have
similar importance. Such spaces, $G_\alpha(\mathbb{R}^d_+)$,
$G^\beta(\mathbb{R}^d_+)$ and $G_\alpha^\beta(\mathbb{R}^d_+)$,
$\alpha+\beta\geq2$ (referred to as $G$-type spaces) in the one
dimensional case $d=1$, were introduced by A. Duran in \cite {D2}
in order to extend the Hankel-Clifford transform, analogous to the
Fourier transform for the positive real line, to a class of
functionals larger than that of tempered distributions on the
positive real line.

In the first part, we are interested in the spaces
$G_\alpha^\alpha(\mathbb{R}^d_+)$, $\alpha\geq 1$,  invariant
under the Hankel-Clifford transform. In the case $d=1$ these
spaces have been characterised through the Laguerre expansions and
the corresponding estimates of the coefficients in \cite{D1}. Our
investigations are connected with the papers of A. Duran
\cite{D4}-\cite{D1} in the case $d=1.$ These papers contain a lot
of fine results but, however, there exist subtle gaps which we
improved upon. Let us briefly present our investigations. In
Section 3, following \cite{D2}, we introduce the spaces
$G^{\beta}_{\alpha}(\RR^d_+)$, $\alpha,\beta\geq 0$. In Section 4,
we improve upon the gaps in \cite{D3} concerning the
Hankel-Clifford transform as a continuous mapping from $\mathcal
S(\R_+)$ into the same space, defining this transform for the
$d$-dimensional case. Moreover, we introduce, as an important
novelty of the paper, the modified fractional powers of the
partial Hankel-Clifford transform as a main tool for an
examination of the $G$-type spaces. In Section 5, the expansion of
elements from $G_\alpha^\alpha(\mathbb{R}^d_+)$, $\alpha\geq 1$,
with respect to the Laguerre orthonormal basis, is presented. We
characterise these spaces through the coefficient estimates. The
main corrections of gaps in \cite{D1}, related to the analytic
function $F(w)$, $w\in\mathbf{D}$, as well as the equivalence of
the conditions of Proposition 5.4, are done. We underline, as an
important novelty of our paper, the topological structure
described in Section 6, since the explanation of this structure in
the case $d=1$ given in \cite{D1} is inadequate. This is
essentially improved in the multi-dimensional case (Theorems 6.1,
6.2) by the closed graph De Wilde theorem. Moreover, as a main
consequence of the analysed topological structure, we prove in
Section 6 the nuclearity of the spaces $G^\alpha_\alpha(\R^d_+)$,
$\alpha\geq 1$ as well as the Schwartz's kernel theorem,
$$G_\alpha^\alpha(\mathbb{R}^{d_1}_+)\hat{\otimes}
G_\alpha^\alpha(\mathbb{R}^{d_2}_+)\cong
G_\alpha^\alpha(\mathbb{R}^{d_1+d_2}_+), \alpha\geq 1.$$

In the second part, we use the obtained series expansions in order
to introduce a new class of pseudo-differential operators with
radial symbols and prove continuity properties of such operators
on the Gelfand- Shilov spaces and their duals. More precisely, we
prove the continuity of the Weyl pseudo-differential operators
with radial symbols from the spaces
$G_{2\alpha}^{2\alpha}(\mathbb{R}^d_+)$ and
$(G_{2\alpha}^{2\alpha}(\mathbb{R}^d_+))'$, $\alpha\geq 1/2$. In
the first case, we show that the class of the Weyl
pseudo-differential operators with radial symbols is a continuous
and linear mapping from $S_\alpha^\alpha(\mathbb{R}^d)$ into
$S_\alpha^\alpha(\mathbb{R}^d)$ which can be extended to a
continuous and linear mapping from
$(S_\alpha^\alpha(\mathbb{R}^d))'$ into
$S_\alpha^\alpha(\mathbb{R}^d)$, while in the second case of the
symbol, we show that the class of the Weyl pseudo-differential
operator is a continuous and linear mapping from
$S_\alpha^\alpha(\mathbb{R}^d)$ into
$S_\alpha^\alpha(\mathbb{R}^d)$ which can be extended to a
continuous and linear mapping from
$(S_\alpha^\alpha(\mathbb{R}^d))'$ into
$(S_\alpha^\alpha(\mathbb{R}^d))'$. This second case is especially
important since we have symbols in the dual spaces as well as the
corresponding mapping over the duals of the Gelfand-Shilov spaces.

As a consequence, in Section 7, we give the corresponding results
related to the symbols in $\mathcal S(\R^d_+)$ and its dual and
the corresponding continuous linear mappings related to  the
Schwartz space $\mathcal S(\R^d)$ and its dual. With these special
cases, we extend the corresponding results of M. W. Wong
\cite[Chapter 24]{Wong}.

 \section{Preliminaries}

We denote by $\NN$, $\ZZ$, $\RR$ and $\CC$ the sets of positive
integers, integers, real and complex numbers, respectively;
$\NN_0=\NN \cup\{0\}$, $\RR_+=(0,\infty)$. The symbol $\RR^d_+$
stands for $(0,\infty)^d$ and $\overline{\RR^d_+}$ for its
closure, i.e. $[0,\infty)^d$. We use the standard multi-index
notation. We denote by $\mathbf{1}=(1,\ldots,1)\in\NN^d$. Thus,
for $z\in\CC^d$, $z^{\mathbf{1}}$ stands for $z_1\cdot\ldots\cdot
z_d$. Moreover, for $x,t\in\overline{\RR^d_+}$,
$x^t=x_1^{t_1}\cdot\ldots \cdot x_d^{t_d}$; in this case we use
the convention $0^0=1$. For $n\in\NN^d_0$, $D^n$ stands for
$\partial^n/\partial t_1^{n_1}\ldots
\partial t_d^{n_d}$. We use $\|\cdot\|_2$ for the norm in the Banach space (abbreviated as a $(B)$-space)
$L^2(\RR^d_+)$.\\
\indent We denote by $\sss$ the space of all complex sequences
$\{a_n\}_{n\in\NN^d_0}$ such that for each $j\in\NN$,
$\ds\sup_{n\in\NN^d_0}|a_n|(|n|+1)^j<\infty$. With these
seminorms, $\sss$ becomes a nuclear Fr\'{e}chet space; from now on
abbreviated as an $(FN)$-space (see \cite[p. 527]{Tr}; clearly in
the definition of $\sss$, we can take the $l^{p}$-norms, $p\geq 1$
instead of the $\mathrm{sup}$-norm). Its strong dual, which we
denote by $\sss'$, consists of all complex valued sequences
$\{a_n\}_{n\in\NN^d_0}$ such that $\ds\sup_{n\in\NN^d_0}
|a_n|(|n|+1)^{-j}<\infty$ for
some $j\in\NN$ (which depends on $\{a_n\}_{n\in\NN^d}$). It is a $(DFN)$-space.\\
\indent Let $\alpha\geq 1$ and $a>1$. We define $\sss^{\alpha,a}$
to be the space of all complex sequences $\{a_n\}_{n\in\NN^d_0}$
for which $\ds \|\{a_n\}_{n\in\NN^d_0}\|_{\sss^{\alpha,a}}=
\sup_{n\in\NN^d_0}|a_n|a^{|n|^{1/\alpha}}<\infty$. With this norm
$\sss^{\alpha,a}$ becomes a $(B)$-space. For $a>b>1$,
$\sss^{\alpha,a}$ is continuously injected into $\sss^{\alpha,b}$.
As a locally convex space (abbreviated as an l.c.s.) we
define $\ds \sss^{\alpha}=\lim_{\substack{\longrightarrow\\
a\rightarrow 1^+}}\sss^{\alpha,a}$; the inductive limit is indeed
a (Hausdorff) l.c.s. since $\sss^{\alpha,a}$ are continuously
injected into $\sss$.

\begin{prop}\label{lkrrr}
For $a>b>1$, the canonical inclusion $\sss^{\alpha,a}\rightarrow
\sss^{\alpha,b}$ is nuclear. In particular, $\sss^{\alpha}$ is a
nuclear $(DFS)$-space (i.e. a $(DFN)$-space) and its strong dual
$(\sss^{\alpha})'$ is an $(FN)$-space.
\end{prop}

\begin{proof} Since the canonical inclusion $\sss^{\alpha,a}\rightarrow \sss^{\alpha,b}$ is a composition of two inclusions of
the same type it is enough to prove that it is quasi-nuclear (for
the definition of a quasi-nuclear mapping see \cite[Definition
3.2.3., p. 56]{pietsch} and for the fact that the composition of
two quasi-nuclear mappings is nuclear see \cite[Theorem 3.3.2., p.
62]{pietsch}). For each $m\in\NN^d_0$, we define
$e_m\in(\sss^{\alpha,a})'$ by $\langle
e_m,\{a_n\}_{n\in\NN^d_0}\rangle= a_m b^{|m|^{1/\alpha}}$. One
easily verifies that $\|e_m\|_{(\sss^{\alpha,a})'}\leq
(b/a)^{|m|^{1/\alpha}}$, hence $\sum_{m\in\NN^d_0}
\|e_m\|_{(\sss^{\alpha,a})'}<\infty$. For
$\{a_n\}_{n\in\NN^d_0}\in\sss^{\alpha,a}$ we have \beas
\|\{a_n\}_{n\in\NN^d_0}\|_{\sss^{\alpha,b}}\leq
\sum_{m\in\NN^d_0}|a_m|b^{|m|^{1/\alpha}}=
\sum_{m\in\NN^d_0}|\langle e_m,\{a_n\}_{n\in\NN^d_0}\rangle|,
\eeas i.e. the canonical inclusion $\sss^{\alpha,a}\rightarrow
\sss^{\alpha,b}$ is quasi nuclear.\qed
\end{proof}

For the moment, denote by $\tilde{\sss}^{\alpha}$ the space of all
complex valued sequences $\{b_n\}_{n\in\NN^d_0}$ such that for
each $a>1$,
$\|\{b_n\}_{n\in\NN^d_0}\|_{\tilde{\sss}^{\alpha},a}=\sum_{n\in\NN^d_0}
|b_n|a^{-|n|^{1/\alpha}}<\infty$. With these seminorms
$\tilde{\sss}^{\alpha}$ becomes an $(F)$-space. Denote by $\Xi$
the mapping $\tilde{\sss}^{\alpha} \rightarrow (\sss^{\alpha})'$,
$\langle \Xi(\{b_n\}_n),\{a_n\}_n\rangle=\sum_n a_n b_n$. One
easily verifies that it is a well defined bijection. Let
$B\subseteq \tilde{\sss}^{\alpha}$ be bounded. If $B_1\subseteq
\sss^{\alpha}$ is bounded, there exists $a>1$ such that
$B_1\subseteq \sss^{\alpha,a}$ and it is bounded there
($\sss^{\alpha}$ is a $(DFN)$-space). Now one easily verifies that
$\ds \sup_{\{b_n\}_n\in B,\, \{a_n\}_n\in B_1}|\langle
\Xi(\{b_n\}_n),\{a_n\}_n\rangle|<\infty$, i.e. $\Xi$ maps bounded
sets into bounded. Since $\tilde{\sss}^{\alpha}$ and
$(\sss^{\alpha})'$ are an $(F)$-spaces, $\Xi$ is continuous and
now the open mapping theorem verifies that $\Xi$ is an
isomorphism. Hence, we proved the following result.

\begin{prop}
The strong dual $(\sss^{\alpha})'$ of $\sss^{\alpha}$ is an
$(FN)$-space of all complex valued sequences
$\{b_n\}_{n\in\NN^d_0}$ such that, for each $a>1$,
$\|\{b_n\}_{n\in\NN^d_0}\|_{(s^{\alpha})',a}=\sum_{n\in\NN^d_0}
|b_n|a^{-|n|^{1/\alpha}}<\infty$. Its topology is generated by the
system of seminorms $\|\cdot\|_{(\sss^{\alpha})',a}$.
\end{prop}

Let $\alpha\geq 1/2$. For $A>0$, denote by
$\SSS^{\alpha,A}_{\alpha,A}(\RR^d)$ a $(B)$-space of all
$\varphi\in\mathcal{C}^{\infty}(\RR^d)$ with the norm
$\ds\sup_{n,m\in\NN^d_0}\left\|x^m
D^n\varphi(x)\right\|_{L^2(\RR^d)}/
(A^{|n|+|m|}n!^{\alpha}m!^{\alpha})<\infty$. The Gelfand-Shilov
space $\SSS^{\alpha}_{\alpha}(\RR^d)$ is defined to be the space
$\ds\lim_{\substack{\longrightarrow\\ A\rightarrow \infty}}
\SSS^{\alpha,A}_{\alpha,A}(\RR^d)$. One easily verifies that for
$A_1<A_2$ the canonical inclusion
$\SSS^{\alpha,A_1}_{\alpha,A_1}(\RR^d)\rightarrow
\SSS^{\alpha,A_2}_{\alpha,A_2}(\RR^d)$ is a compact mapping, i.e.
$\SSS^{\alpha}_{\alpha}(\RR^d)$ is a $(DFS)$-space (for the
properties of
$\SSS^{\alpha}_{\alpha}(\RR^d)$ we refer to \cite[Chapter 6]{NR}; see also \cite{GS1}, \cite{GPR2}).\\
\indent The Hermite polynomials and the corresponding Hermite
functions are given by
$$H_j(t)=(-1)^je^{t^2}\frac{d^j}{dt^j}(e^{-t^2}),\;t\in\mathbb{R},\;
h_j(t)=(2^j
j!\sqrt{\pi})^{-1/2}e^{-t^2/2}H_j(t),\;t\in\mathbb{R},\;j\in\mathbb{N}_0.$$
For $n\in\NN^d_0$, put $h_n(x)=h_{n_1}(x_1)\cdot\ldots\cdot
h_{n_d}(x_d)$. Then $\{h_n\}_{n\in\NN^d_0}$ is an orthonormal
basis for $L^2(\mathbb{R}^d)$. For each $n\in\NN^d_0$,
$h_n\in\SSS^{1/2}_{1/2}(\RR^d)$. Moreover, for $\alpha\geq 1/2$,
$S_\alpha^\alpha(\mathbb{R}^d)$ is  given through the Hermite
expansions. In fact, we have the following result for which the
proof is similar to the proof of \cite[Theorem 3.4 and Corollary
3.5]{lan} and we omit it.

\begin{prop}\label{herexpult}
Let $\alpha\geq 1/2$. The mapping
$\SSS^{\alpha}_{\alpha}(\RR^d)\rightarrow \sss^{2\alpha}$,
$f\mapsto \{\langle f,h_n\rangle\}_{n\in\NN^d_0}$, is a
topological isomorphism. For $f\in\SSS^{\alpha}_{\alpha}(\RR^d)$,
$\sum_{n\in\NN^d_0}\langle f,h_n\rangle h_n$ converges absolutely to $f$ in $\SSS^{\alpha}_{\alpha}(\RR^d)$.\\
\indent The mapping $(\SSS^{\alpha}_{\alpha}(\RR^d))'\rightarrow
(\sss^{2\alpha})'$, $T\mapsto \{\langle
T,h_n\rangle\}_{n\in\NN^d_0}$, is a topological isomorphism. For
$T\in(\SSS^{\alpha}_{\alpha}(\RR^d))'$, $\sum_{n\in\NN^d_0}\langle
T,h_n\rangle h_n$ converges absolutely to $T$ in
$(\SSS^{\alpha}_{\alpha}(\RR^d))'$.
\end{prop}

For $j\in\mathbb{N}_0$ and $\gamma\geq 0$, the $j$-th Laguerre
polynomial of order $\gamma$ is defined by
$$L_j^\gamma(t)=\frac{t^{-\gamma}e^t}{j!}\frac{d^j}{dt^j}(e^{-t}t^{\gamma+j}),\;t \geq 0.$$
The $j$-th Laguerre function of order $\gamma$ is defined by
$\mathcal{L}^{\gamma}_j(t)=\left(j!/\Gamma(j+\gamma+1)\right)^{1/2}L_j^{\gamma}(t)e^{-t/2}$.
For $n\in\NN^d_0$ and $\gamma\in\overline{\RR^d_+}$,
$L^{\gamma}_n(t)=\prod_{l=1}^d L^{\gamma_l}_{n_l}(t_l)$ and
$\mathcal{L}^{\gamma}_n(t)=\prod_{l=1}^d\mathcal{L}^{\gamma_l}_{n_l}(t_l)$
are the $d$-dimensional Laguerre polynomials and Laguerre
functions of order $\gamma$, respectively.
In the case $\gamma=0$, we write $L_n$ and $\mathcal{L}_n$ instead of $L_n^0$ and $\mathcal{L}_n^0$, respectively.\\
\indent For $\gamma\in\overline{\RR^d_+}$ we denote by
$L^2(\RR^d_+,t^{\gamma}dt)$ a $(B)$-space of all measurable
functions on $\RR^d_+$ such that $\int_{\RR^d_+}
|f(t)|^2t^{\gamma}dt<\infty$; its norm is defined by the square
root of the last quantity. Moreover,
$\{\mathcal{L}^{\gamma}_n\}_{n\in\NN^d_0}$ is an orthonormal basis for $L^2(\mathbb{R}^d_+,t^{\gamma}dt)$.\\
\indent Recall, (see \cite{Sm}) an $(F)$-space
$\SSS({\mathbb{R}^d_+})$ consists of all
$f\in\mathcal{C}^\infty(\mathbb{R}^d_+)$ such that all derivatives
$D^pf$, $p\in\NN^d_0$, extend to continuous functions on
$\overline{\RR^d_+}$ and
$\sup_{x\in\mathbb{R}^d_+}x^k|D^pf(x)|<\infty$, $\forall
k,p\in\mathbb{N}_0^d$. We denote by $(\SSS(\mathbb{R}_+^d))'$ its
strong dual.

In \cite{D4}, \cite{Zayed} and \cite{SP1} the expansions of the
functions from $\SSS({\mathbb{R}^d_+})$ with respect to the
Laguerre polynomials are studied. The $d$-dimensional case is
considered in \cite{Sm}. We state these results here and refer to
\cite{Sm} for their proofs.

\begin{thm}(\cite[Theorem 3.1]{Sm})
For $f\in\mathcal{S}(\mathbb{R}^d_+)$ let
$a_n(f)=\int_{\mathbb{R}^d_+} f(x)\mathcal{L}_n(x)dx$. Then
$f=\sum_{n\in\mathbb{N}_0^d}a_n(f)\mathcal{L}_n$ and the series
converges absolutely in $\mathcal{S}(\mathbb{R}^d_+)$. Moreover
the mapping $\iota:\mathcal{S}(\mathbb{R}^d_+)\rightarrow \sss$,
$\iota(f)=\{a_n(f)\}_{n\in\NN^d_0}$ is a topological isomorphism.
\end{thm}

\begin{thm}\label{razvoj u S+'}(\cite[Theorem 3.2]{Sm})
For $T\in\mathcal{S}'(\mathbb{R}^d_+)$ let $b_n(T)=\langle
T,\mathcal{L}_n\rangle$. Then $T=\sum_{n\in\mathbb{N}_0^d}b_n(T)$
$\mathcal{L}_n$ and $\{b_n(T)\}_{n\in\mathbb{N}_0^d}\in \sss'$ and
the series converges absolutely in $\mathcal{S}'(\mathbb{R}^d_+)$.
Conversely, if $\{b_n\}_{n\in\mathbb{N}_0^d}\in \sss'$ there
exists $T\in\mathcal{S}'(\mathbb{R}^d_+)$ such that
$T=\sum_{n\in\mathbb{N}_0^d}b_n\mathcal{L}_n$. As a consequence,
$\mathcal{S}'(\mathbb{R}^d_+)$ is topologically isomorphic to
$\sss'$.
\end{thm}

Note that the topological isomorphisms between $\SSS(\RR^d_+)$ and
$\sss$ and between $(\SSS(\RR^d_+))'$ and $\sss'$ imply that
$\SSS(\RR^d_+)$ is an $(FN)$-space and $(\SSS(\RR^d_+))'$ is a
$(DFN)$-space. In particular, the $\pi$ and the $\epsilon$
topologies coincide on $\SSS(\RR^{d_1}_+)\otimes\SSS(\RR^{d_2}_+)$
and on $(\SSS(\RR^{d_1}_+))'\otimes (\SSS(\RR^{d_2}_+))'$.

\begin{thm}\label{kerth}(\cite[Theorem 4.2]{Sm})
The following canonical isomorphisms hold:
$$\mathcal{S}(\mathbb{R}^{d_1}_+)\hat{\otimes}\mathcal{S}(\mathbb{R}^{d_2}_+)\cong
\mathcal{S}(\mathbb{R}^{d_1+d_2}_+),\,\,\,
(\mathcal{S}(\mathbb{R}^{d_1}_+))'\hat{\otimes}(\mathcal{S}(\mathbb{R}^{d_2}_+))'\cong
(\mathcal{S}(\mathbb{R}^{d_1+d_2}_+))'.$$
\end{thm}

\begin{thm}\label{repofsupp}(\cite[Theorem 4.3]{Sm})
The restriction mapping $f\mapsto f_{|\RR^d_+}$, $\SSS(\RR^d)\rightarrow \SSS(\RR^d_+)$ is a topological homomorphism onto.\\
\indent The space $\SSS(\RR^d_+)$ is topologically isomorphic to
the quotient space $\SSS(\RR^d)/N$, where $N=\{f\in\SSS(\RR^d)|\,
\mathrm{supp}\, f\subseteq \RR^d\backslash \RR^d_+\}$.
Consequently, $(\SSS(\RR^d_+))'$ can be identified with the closed
subspace of $(\SSS(\RR^d))'$ which consists of all tempered
distributions with support in $\overline{\RR^d_+}$.
\end{thm}

\begin{rem}\label{derinss}
The fact that $(\SSS(\RR^d_+))'$ is canonically isomorphic to the
closed subspace of $(\SSS(\RR^d))'$ which consists of all tempered
distributions with support in $\overline{\RR^d_+}$ allows us to
define unambiguously the notion of derivatives of the elements of
$(\SSS(\RR^d_+))'$. In fact, for $T\in(\SSS(\RR^d_+))'$ and
$n\in\NN^d_0$, $D^nT$ stands for the $D^n$-derivative of $T$ in
$(\SSS(\RR^d))'$ sense. Since $\mathrm{supp}\, D^nT\subseteq
\overline{\RR^d_+}$, $D^nT$ is a well defined element of
$(\SSS(\RR^d_+))'$. Moreover, by $\SSS(\RR^d_+)\cong\SSS(\RR^d)/N$
(see Theorem \ref{repofsupp}) \beas \langle
D^nT,\varphi\rangle=(-1)^{|n|}\langle T,D^n\varphi\rangle,\,\,\,
\forall\varphi\in\SSS(\RR^d_+). \eeas It is important to stress
that if $T$ is given by $\psi\in\SSS(\RR^d_+)$ then $D^nT$ does
not have to coincide with the classical $D^n$-derivative of $\psi$
(unless $\psi$ can be extended to a smooth function on $\RR^d$
with support in $\overline{\RR^d_+}$). Considering $\psi$ as an
element of $(\SSS(\RR^d_+))'$ automatically means extending it by
$0$ on $\RR^d\backslash\overline{\RR^d_+}$. Of course, this
extension does not have to be smooth.
\end{rem}

\section{Definition and basic properties of the spaces $G^{\beta}_{\alpha}(\RR^d_+)$, $\alpha,\beta\geq 0$}

Unless otherwise stated, $\alpha$ and $\beta$ are two positive
reals. In the sequel, we will often tacitly apply the following
inequalities (we use $0^0=1$) \beas m^{\alpha m}n^{\alpha n}\leq
(m+n)^{\alpha(m+n)},\,\, (m+n)^{\alpha(m+n)}\leq
e^{\alpha|m+n|}m^{\alpha m}n^{\alpha n},\,\, \forall
m,n\in\NN^d_0,\,\, \forall \alpha\geq 0. \eeas
We define the basic test spaces (cf. \cite[Definition 2.1]{D2} for $d$=1):\\
Let $A>0$. We denote by $G^{\beta,A}_{\alpha,A}(\RR^d_+)$ the
space of all $f\in\SSS(\RR^d_+)$ for which \beas
\sup_{p,k\in\NN^d_0}\frac{\|t^{(p+k)/2}D^pf(t)\|_2}
{A^{|p+k|}k^{(\alpha/2)k}p^{(\beta/2)p}}<\infty. \eeas With the
following seminorms \beas
\sigma_{A,j}(f)=\sup_{p,k\in\NN^d_0}\frac{\|t^{(p+k)/2}D^pf(t)\|_{L^2(\RR^d_+)}}
{A^{|p+k|}k^{(\alpha/2)k}p^{(\beta/2)p}}+\sup_{\substack{|p|\leq
j\\ |k|\leq j}} \sup_{t\in\RR^d_+}|t^k D^pf(t)|,\,\, j\in\NN_0,
\eeas one easily verifies that it becomes an $(F)$-space. Clearly,
if $A_1< A_2$, $G^{\beta,A_1}_{\alpha,A_1}(\RR^d_+)$ is
continuously injected into $G^{\beta,A_2}_{\alpha,A_2}(\RR^d_+)$.
Define $\ds G^{\beta}_{\alpha}(\RR^d_+)=
\lim_{\substack{\longrightarrow\\ A\rightarrow \infty}}
G^{\beta,A}_{\alpha,A}(\RR^d)$. Since all the injections
$G^{\beta,A}_{\alpha,A}\rightarrow \SSS(\RR^d_+)$ are continuous,
$G^{\beta}_{\alpha}(\RR^d_+)$ is indeed a (Hausdorff) l.c.s..
Clearly, $G^{\beta}_{\alpha}(\RR^d_+)$ is continuously injected
into
$\SSS(\RR^d_+)$. As inductive limit of an $(F)$-spaces, $G^{\beta}_{\alpha}(\RR^d_+)$ is a barrelled and bornological l.c.s..\\
\indent For $A>0$ we define $G_{\alpha,A}(\RR^d_+)$ to be the
space of all $f\in\SSS(\RR^d_+)$ such that \beas
\sup_{k\in\NN^d_0}\frac{\|t^{(p+k)/2}D^pf(t)\|_2}
{A^{|k|}k^{(\alpha/2)k}}<\infty,\,\, \forall p\in\NN^d_0 \eeas and
similarly, $G^{\beta,A}(\RR^d_+)$ to be the space of all
$f\in\SSS(\RR^d_+)$ such that \beas
\sup_{p\in\NN^d_0}\frac{\|t^{(p+k)/2}D^pf(t)\|_2}
{A^{|p|}p^{(\beta/2)p}}<\infty,\,\, \forall k\in\NN^d_0. \eeas If
we equip $G_{\alpha,A}(\RR^d_+)$ with the system of seminorms
\beas \sigma'_{A,j}(f)=\sup_{|p|\leq
j}\sup_{k\in\NN^d_0}\frac{\|t^{(p+k)/2}D^pf(t)\|_{L^2(\RR^d_+)}}
{A^{|k|}k^{(\alpha/2)k}}+\sup_{\substack{|p|\leq j\\ |k|\leq j}}
\sup_{t\in\RR^d_+}|t^k D^pf(t)|,\,\, j\in\NN_0, \eeas one easily
verifies that it becomes an $(F)$-space. Analogously, by equipping
$G^{\beta,A}(\RR^d_+)$ with the system of seminorms \beas
\sigma''_{A,j}(f)=\sup_{|k|\leq j}
\sup_{p\in\NN^d_0}\frac{\|t^{(p+k)/2}D^pf(t)\|_{L^2(\RR^d_+)}}
{A^{|p|}p^{(\beta/2)p}}+\sup_{\substack{|p|\leq j\\ |k|\leq j}}
\sup_{t\in\RR^d_+}|t^k D^pf(t)|,\,\, j\in\NN_0, \eeas it is also
an $(F)$-space. Similarly as above, we define $\ds
G_{\alpha}(\RR^d_+)=\lim_{\substack{\longrightarrow\\ A\rightarrow
\infty}} G_{\alpha,A}(\RR^d_+)$ and $\ds
G^{\beta}(\RR^d_+)=\lim_{\substack{\longrightarrow\\ A\rightarrow
\infty}}G^{\beta,A}(\RR^d_+)$. Thus, $G_{\alpha}(\RR^d_+)$ and
$G^{\beta}(\RR^d_+)$ are barrelled and bornological l.c.s. that are continuously injected into $\SSS(\RR^d_+)$.\\
\indent For each $m\in\NN^d_0$, $f(t)\mapsto t^m f(t)$ is a
continuous mapping $G_\alpha(\mathbb{R}^d_+)\rightarrow
G_{\alpha}(\RR^d_+)$, $G^\beta(\mathbb{R}^d_+)\rightarrow
G^{\beta}(\RR^d_+)$ and
$G_\alpha^\beta(\mathbb{R}^d_+)\rightarrow G^{\beta}_{\alpha}(\RR^d_+)$.\\
\indent We denote by $(G^{\beta}(\RR^d_+))'$,
$(G_{\alpha}(\RR^d_+))'$ and $(G^{\beta}_{\alpha}(\RR^d_+))'$ the
strong duals
of $G^{\beta}(\RR^d_+)$, $G_{\alpha}(\RR^d_+)$ and $G^{\beta}_{\alpha}(\RR^d_+)$, respectively.\\
\indent One easily verifies that when $\alpha,\beta\geq 1$,
$\mathcal{L}_n\in G^{\beta}_{\alpha}(\RR^d_+)$ and hence
$G^{\beta}_{\alpha}(\RR^d_+)$ is dense in $\SSS(\RR^d_+)$. In
particular, for $\alpha\geq 1$, $G_{\alpha}(\RR^d_+)$,
$G^{\alpha}(\RR^d_+)$ and $G^{\alpha}_{\alpha}(\RR^d_+)$ are dense
in $\SSS(\RR^d_+)$. Hence, $(\SSS(\RR^d_+))'$ is continuously
injected into $(G_{\alpha}(\RR^d_+))'$, $(G^{\alpha}(\RR^d_+))'$
and $(G^{\alpha}_{\alpha}(\RR^d_+))'$.

\begin{rem}
Let $\alpha,\beta>0$. Then the spaces
$G_{\alpha}^\beta(\mathbb{R}^d_+)$ are non-trivial when
$\alpha+\beta\geq 2$. We refer to \cite[Corollary 3.9]{D2} for
d=1. For $d$-dimensional case it follows considering the function
$\varphi(t)=\varphi_1(t_1)\ldots \varphi_d(t_d)$, where
$\varphi_j$, $j=1,\ldots,d$, is a non-zero element of
$G^{\beta}_{\alpha}(\RR_+)$.
\end{rem}

\section{The Hankel-Clifford transform}

Let $\mathcal{C}_{L^{\infty}}(\overline{\RR^d_+})$ be a
$(B)$-space of all continuous functions
$f:\overline{\RR^d_+}\rightarrow \CC$ such that
$\ds\sup_{x\in\overline{\RR^d_+}}|f(x)|<\infty$; the norm of $f\in\mathcal{C}_{L^{\infty}}(\overline{\RR^d_+})$ is given by the left-hand side.\\
\indent For $\gamma\geq 0$, we denote by $J_{\gamma}$ and
$I_{\gamma}$ the Bessel function of the first kind and the
modified Bessel function of the first kind, respectively. Denote
$\mathbf{T}^{(d)}=\{z\in\CC^d|\, |z_l|=1,\, z_l\neq 1,\, \forall
l=1,\ldots,d\}$. For $z\in\mathbf{T}^{(d)}$ and
$\gamma\in\overline{\RR^d_+}$, we define the fractional powers and
the modified fractional powers of the Hankel-Clifford transform of
$f\in \SSS(\mathbb{R}_+^d)$ by \beas
\mathcal{I}_{z,\gamma}f(t)&=&\left(\prod_{l=1}^d(1-z_l)^{-1}e^{-\frac{1}{2}\frac{1+z_l}{1-z_l}t_l}\right)\int_{\RR^d_+}f(x)\prod_{l=1}^d
e^{-\frac{1}{2}
\frac{1+z_l}{1-z_l}x_l} (x_lt_lz_l)^{-\gamma_l/2}x_l^{\gamma_l}I_{\gamma_l}\left(\frac{2\sqrt{x_lt_lz_l}}{1-z_l}\right)dx\\
\mathcal{J}_{z,\gamma}f(t)&=&\left(\prod_{l=1}^d(1-z_l)^{-1}\right)\int_{\RR^d_+}f(x)\prod_{l=1}^d
(x_lt_lz_l)^{-\gamma_l/2}x_l^{\gamma_l}I_{\gamma_l}\left(\frac{2\sqrt{x_lt_lz_l}}{1-z_l}\right)dx.
\eeas Since $z\in\mathbf{T}^{(d)}$, $z_l=e^{i\theta_l}$ where
$\theta_l\in(-\pi,\pi]\backslash\{0\}$, $l=1,\ldots,d$. Observe
that $(1+z_l)/(1-z_l)$ is purely imaginary. Moreover,
$2\sqrt{x_lt_lz_l}/(1-z_l)=i\sqrt{x_lt_l}/\sin(\theta_l/2)$ and
$(x_lt_lz_l)^{-\gamma_l/2}=(x_lt_l)^{-\gamma_l/2}e^{-i\theta_l\gamma_l/2}$.
Hence, for $l=1,\ldots,d$, \bea\label{equfb}
(x_lt_lz_l)^{-\gamma_l/2}I_{\gamma_l}\left(\frac{2\sqrt{x_lt_lz_l}}{1-z_l}\right)=e^{-i\theta_l\gamma_l/2}(x_lt_l)^{-\gamma_l/2}
e^{(i\gamma_l\pi\mathrm{sgn}\,
\theta_l)/2}J_{\gamma_l}\left(\frac{\sqrt{x_lt_l}}{|\sin(\theta_l/2)|}\right).
\eea By the definition of the Bessel function of the first kind,
it is clear that for $\nu\geq 0$, $\xi^{-\nu}|J_{\nu}(\xi)|$ is
uniformly bounded when $\xi\in(0,c)$ for arbitrary but fixed
$c\geq 1$. Combining this with \cite[9.2.1, p. 364]{abrsteg}, we
obtain that there exists $C\geq 1$ such that \bea\label{equfb11}
\left|\prod_{l=1}^d
(x_lt_lz_l)^{-\gamma_l/2}I_{\gamma_l}\left(\frac{2\sqrt{x_lt_lz_l}}{1-z_l}\right)\right|\leq
C,\,\, \forall x,t\in \RR^d_+. \eea Moreover, for $\nu\geq0$, by
the definition of $J_{\nu}$, the function $\xi\mapsto
\xi^{-\nu}J_{\nu}(\xi)$, $\RR_+\rightarrow \CC$, can be extended
to a continuous function on $\overline{\RR_+}$. Hence,
(\ref{equfb}) and (\ref{equfb11}) imply that for
$f\in\SSS(\RR^d_+)$ the integrals in the definition for
$\mathcal{I}_{z,\gamma}f$ and $\mathcal{J}_{z,\gamma} f$ converge
absolutely i.e. $\mathcal{I}_{z,\gamma}f,\mathcal{J}_{z,\gamma}
f\in\mathcal{C}_{L^{\infty}}(\overline{\RR^d_+})$. when
$f_j\rightarrow f$ in $\SSS(\RR^d_+)$,
$\mathcal{I}_{z,\gamma}f_j\rightarrow \mathcal{I}_{z,\gamma}f$ and
$\mathcal{J}_{z,\gamma}f_j\rightarrow \mathcal{J}_{z,\gamma}f$ in
$\mathcal{C}_{L^{\infty}}(\overline{\RR^d_+})$. Hence,
$\mathcal{I}_{z,\gamma}$ and $\mathcal{J}_{z,\gamma}$ are well
defined continuous mappings from $\SSS(\RR^d_+)$ to
$\mathcal{C}_{L^{\infty}}(\overline{\RR^d_+})$. Our goal is to
prove that $\mathcal{I}_{z,\gamma}$ and $\mathcal{J}_{z,\gamma}$
are continuous mappings from $\SSS(\RR^d_+)$ to $\SSS(\RR^d_+)$.
Firstly, we prove this for $\mathcal{J}_{z,\gamma}$ in the case
$d=1$.

\begin{lem}\label{ttt111355}
For $z\in\mathbf{T}^{(1)}$ and $\gamma\geq 0$,
$\mathcal{J}_{z,\gamma}$ is a continuous mapping from
$\SSS(\RR_+)$ into $\SSS(\RR_+)$.
\end{lem}

\begin{proof} Clearly $\SSS(\RR_+)$ is continuously injected into $L^2(\RR_+,t^{\gamma}dt)$. Let $E_{\gamma}$ be the operator
\beas
E_{\gamma}=tD^2+D-\frac{t}{4}-\frac{\gamma^2}{4t}+\frac{\gamma+1}{2}=D(tD)-\frac{t}{4}-\frac{\gamma^2}{4t}+\frac{\gamma+1}{2}.
\eeas Then
$E_{\gamma}(t^{\gamma/2}\mathcal{L}^{\gamma}_n(t))=-nt^{\gamma/2}\mathcal{L}^{\gamma}_n(t)$
(see \cite[(11), p. 188]{Ed}). For $f\in \SSS(\RR_+)$ we have
\beas
E_{\gamma}(t^{\gamma/2}f(t))=t^{\gamma/2}\left((\gamma+1)Df(t)+tD^2f(t)-tf(t)/4+(\gamma+1)f(t)/2\right).
\eeas Hence, for $k\in\NN$,
$E^k_{\gamma}(t^{\gamma/2}f(t))=t^{\gamma/2}g_k(t)$ for some
$g_k\in\SSS(\RR_+)$. Let
$a_n(f)=\int_0^{\infty}f(t)\mathcal{L}^{\gamma}_n(t)t^{\gamma}dt$.
Then, by integration by parts, we have \beas
\int_0^{\infty}g_1(t)\mathcal{L}^{\gamma}_n(t)t^{\gamma}dt=\int_0^{\infty}E_{\gamma}(t^{\gamma/2}f(t))
\mathcal{L}^{\gamma}_n(t)t^{\gamma/2}dt =-na_n(f). \eeas Iterating
this, we obtain \beas
\int_0^{\infty}g_k(t)\mathcal{L}^{\gamma}_n(t)t^{\gamma}dt=(-n)^ka_n(f).
\eeas Since $g_k\in\SSS(\RR_+)\subseteq L^2(\RR_+,t^{\gamma}dt)$,
we conclude $\{a_n(f)\}_{n\in\NN_0}\in\sss$. Observe that
$f=\sum_n a_n(f)\mathcal{L}^{\gamma}_n$ in
$L^2(\RR_+,t^{\gamma}dt)$. We need the following estimate for the
derivatives of the Laguerre polynomials (see \cite[Theorem
1]{durlp}): \beas \left|t^kD^p(e^{-t/2}L_n^{\gamma}(t))\right|\leq
2^{-\min\{\gamma,k\}}4^k(n+1)\cdot\ldots\cdot(n+k)\binom{n+\max\{\gamma-k,0\}+p}
{n}, \eeas for all $t\geq 0$, $n,k,p\in\NN_0$. Denote by
$[\gamma]$ the integral part of $\gamma$, we have \beas
\binom{n+\max\{\gamma-k,0\}+p}{n}\leq
\binom{n+[\gamma]+1+p}{n}\leq (n+[\gamma]+p+1)^{[\gamma]+p+1}.
\eeas Hence, there exists $C_{p,k}\geq 1$ which depends on $p$ and
$k$, but not on $n$, such that \bea\label{estlag333}
\left|t^kD^p\mathcal{L}_n^{\gamma}(t)\right|\leq
C_{p,k}(n+1)^{k+p+[\gamma]+1}. \eea Since $\{a_n(f)\}_n\in\sss$,
we have
$\sum_n|a_n(f)|\sup_{t\in\RR_+}|t^kD^p\mathcal{L}^{\gamma}_n(t)|<\infty$,
i.e. $\sum_na_n(f)\mathcal{L}^{\gamma}_n$ converges absolutely in
$\SSS(\RR_+)$. Since
$\mathcal{J}_{z,\gamma}f:\SSS(\RR_+)\rightarrow
\mathcal{C}_{L^{\infty}}(\overline{\RR_+})$ is continuous,
$\mathcal{J}_{z,\gamma}f=\sum_n
a_n(f)\mathcal{J}_{z,\gamma}\mathcal{L}^{\gamma}_n$ and the series
converges absolutely in
$\mathcal{C}_{L^{\infty}}(\overline{\RR_+})$. We need the
following equality (see \cite[(4.20.3), p. 83]{Le})
\bea\label{eqlag} \int_0^{\infty} J_{\gamma}(\sqrt{xt})
x^{\gamma/2}\mathcal{L}^{\gamma}_n(x)dx=2(-1)^n
t^{\gamma/2}\mathcal{L}^{\gamma}_n(t),\,\, \gamma\geq 0,\,
n\in\NN_0. \eea Using (\ref{equfb}) and (\ref{eqlag}) we obtain
\bea\label{tttss11}
\mathcal{J}_{z,\gamma}\mathcal{L}^{\gamma}_n(t)=2(-1)^{n}
e^{-i\gamma\theta/2}e^{(i\gamma\pi\mathrm{sgn}\,\theta)/2}(1-e^{i\theta})^{-1}|\sin(\theta/2)|^{-\gamma}
\mathcal{L}^{\gamma}_n\left(t/\sin^2(\theta/2)\right). \eea The
estimate (\ref{estlag333}) together with (\ref{tttss11}) implies
that $\sum_na_n(f)\mathcal{J}_{z,\gamma}\mathcal{L}^{\gamma}_n$
converges absolutely in $\SSS(\RR_+)$. Thus, we obtain that the
image of $\SSS(\RR_+)$ under $\mathcal{J}_{z,\gamma}$ is contained
in $\SSS(\RR_+)$. Since
$\mathcal{J}_{z,\gamma}:\SSS(\RR_+)\rightarrow
\mathcal{C}_{L^{\infty}}(\overline{\RR_+})$ is continuous its
graph is closed in $\SSS(\RR_+)\times
\mathcal{C}_{L^{\infty}}(\overline{\RR_+})$. As $\SSS(\RR_+)$ is
continuously injected into
$\mathcal{C}_{L^{\infty}}(\overline{\RR_+})$ and
$\mathcal{J}_{z,\gamma}\left(\SSS(\RR_+)\right)\subseteq
\SSS(\RR_+)$, the graph of $\mathcal{J}_{z,\gamma}$ is closed in
$\SSS(\RR_+)\times\SSS(\RR_+)$. Since $\SSS(\RR_+)$ is an
$(F)$-space, the closed graph theorem implies that
$\mathcal{J}_{z,\gamma}:\SSS(\RR_+)\rightarrow \SSS(\RR_+)$ is
continuous. \qed
\end{proof}

Now, by the principle of induction, we show that for
$z\in\mathbf{T}^{(d)}$ and $\gamma\in\overline{\RR^d_+}$,
$\mathcal{J}_{z,\gamma}$ is a continuous mapping from
$\SSS(\RR^d_+)$ into itself. When $f\in\SSS(\RR^d_+)$, we denote
$\mathcal{J}_{z,\gamma}$ by $\mathcal{J}^{(d)}_{z,\gamma}$ in
order to avoid confusions. We already considered the case $d=1$;
$\mathcal{J}^{(1)}_{z,\gamma}:\SSS(\RR_+)\rightarrow \SSS(\RR_+)$
is continuous. Let $\mathcal{J}^{(d)}_{z,\gamma}$ be continuous.
Let $\nu=(\gamma,\gamma')\in\overline{\RR^{d+1}_+}$ where
$\gamma\in\overline{\RR^d_+}$ and $\gamma'\geq 0$ and let
$\zeta=(z,z')\in \mathbf{T}^{(d+1)}$ where $z\in\mathbf{T}^{(d)}$
and $z'\in\mathbf{T}^{(1)}$. The mapping
$\mathcal{J}^{(d)}_{z,\gamma}\otimes\mathcal{J}^{(1)}_{z',\gamma'}:
\SSS(\RR^d_+)\otimes_{\pi}\SSS(\RR_+)\rightarrow
\SSS(\RR^d_+)\otimes_{\pi}\SSS(\RR_+)$ is continuous. Denoting by
$\tilde{\mathcal{J}}_{\zeta,\nu}$ its continuous extension on the
completions, Theorem \ref{kerth} yields that
$\tilde{\mathcal{J}}_{\zeta,\nu}$ is a continuous mapping from
$\SSS(\RR^{d+1}_+)$ into itself. Observe that for each $f\in
\SSS(\RR^d_+)\otimes\SSS(\RR_+)$,
$\mathcal{J}^{(d+1)}_{\zeta,\nu}f(t)=\tilde{\mathcal{J}}_{\zeta,\nu}f(t)$,
$\forall t\in\RR^{d+1}$. Thus
$\mathcal{J}^{(d+1)}_{\zeta,\nu}f\in\SSS(\RR^{d+1}_+)$. If
$f\in\SSS(\RR^{d+1}_+)$, there exists a sequence $f_j\in
\SSS(\RR^d_+)\otimes\SSS(\RR_+)$, $j\in\NN$, such that
$f_j\rightarrow f$ in $\SSS(\RR^{d+1}_+)$ (cf. Theorem
\ref{kerth}; $\SSS(\RR^{d+1}_+)$ is an $(F)$-space). Since we
proved that
$\mathcal{J}^{(d+1)}_{\zeta,\nu}:\SSS(\RR^{d+1}_+)\rightarrow\mathcal{C}_{L^{\infty}}
(\overline{\RR^{d+1}_+})$ is continuous (see the discussion before
Lemma \ref{ttt111355}, we have, for each fixed $t\in\RR^{d+1}_+$,
\beas \mathcal{J}^{(d+1)}_{\zeta,\nu}f(t)=\lim_{j\rightarrow
\infty}\mathcal{J}^{(d+1)}_{\zeta,\nu}f_j(t)= \lim_{j\rightarrow
\infty}\tilde{\mathcal{J}}_{\zeta,\nu}f_j(t)=\tilde{\mathcal{J}}_{\zeta,\nu}f(t).
\eeas
Hence $\mathcal{J}^{(d+1)}_{\zeta,\nu}f\in\SSS(\RR^{d+1}_+)$ and $\mathcal{J}^{(d+1)}_{\zeta,\nu}f=\tilde{\mathcal{J}}_{\zeta,\nu}f$, $\forall f\in\SSS(\RR^{d+1}_+)$. We conclude that $\mathcal{J}^{(d+1)}_{\zeta,\nu}:\SSS(\RR^{d+1}_+)\rightarrow \SSS(\RR^{d+1}_+)$ is continuous.\\
\indent Next we prove that $\mathcal{J}_{z,\gamma}$ extends to
isometry from
$L^2(\RR^d_+,t^{\gamma}dt)$ onto itself. Firstly, we prove the following claim:\\
\indent For $\gamma\in\overline{\RR^d_+}$, let
$V^{(d)}_{\gamma}\subseteq \SSS(\RR^d_+)$ be the space which
consists of all finite linear combinations of the form
$\sum_{k\leq n}a_k\mathcal{L}^{\gamma}_k$, where $a_k\in\CC$.
Then, for each
$\gamma\in\overline{\RR^d_+}$, $V^{(d)}_{\gamma}$ is dense in $\SSS(\RR^d_+)$.\\
\indent The proof follows by the principle of induction on the
dimension. For $d=1$, it is already proved in the first part of
the proof of Lemma \ref{ttt111355}. Assume that the assertion
holds for $d\in\NN$. Let
$\nu=(\gamma,\gamma')\in\overline{\RR^{d+1}_+}$ where
$\gamma\in\overline{\RR^d_+}$ and $\gamma'\geq 0$. The inductive
hypothesis implies that $V^{(d)}_{\gamma}\otimes
V^{(1)}_{\gamma'}$ is dense in
$\SSS(\RR^d_+)\otimes_{\epsilon}\SSS(\RR_+)$ and consequently in
$\SSS(\RR^{d+1}_+)$ by Theorem \ref{kerth}. One easily verifies
that $V^{(d)}_{\gamma}\otimes V^{(1)}_{\gamma'}\subseteq V^{(d+1)}_{\nu}$ and the proof is completed.\\
\indent By (\ref{equfb}) and (\ref{eqlag}), we obtain
\bea\label{tttss}
\mathcal{J}_{z,\gamma}\mathcal{L}^{\gamma}_n(t)=2^d(-1)^{|n|}
c_{z,\gamma}\left(\prod_{l=1}^d|\sin(\theta_l/2)|^{-\gamma_l}\right)
\mathcal{L}^{\gamma}_n\left(\frac{t_1}{\sin^2(\theta_1/2)},\ldots,
\frac{t_d}{\sin^2(\theta_d/2)}\right), \eea where
$c_{z,\gamma}=\prod_{l=1}^d
e^{-i\gamma_l\theta_l/2}e^{(i\gamma_l\pi\mathrm{sgn}\,\theta_l)/2}(1-e^{i\theta_l})^{-1}$.
One easily verifies that the set
$\{\mathcal{J}_{z,\gamma}\mathcal{L}^{\gamma}_n|\, n\in\NN^d_0\}$
is orthonormal in $L^2(\RR^d_+,t^{\gamma}dt)$. Now, for $f\in
V^{(d)}_{\gamma}$ ($V^{(d)}_{\gamma}$ is a subspace of
$\SSS(\RR^d_+)$ defined in the assertion above) we have
$\|\mathcal{J}_{z,\gamma}f\|_{L^2(\RR^d_+,t^{\gamma}dt)}=\|f\|_{L^2(\RR^d_+,t^{\gamma}dt)}$.
Since $V^{(d)}_{\gamma}$ is dense in $\SSS(\RR^d_+)$ we have
$\|\mathcal{J}_{z,\gamma}f\|_{L^2(\RR^d_+,t^{\gamma}dt)}=\|f\|_{L^2(\RR^d_+,t^{\gamma}dt)}$
for all $f\in\SSS(\RR^d_+)$. Thus $\mathcal{J}_{z,\gamma}$ extends
to an isometry from $L^2(\RR^d_+,t^{\gamma}dt)$ into itself.
Secondly, we prove the surjectivity of $\mathcal{J}_{z,\gamma}$.
Note that
$\mathcal{J}_{\bar{z},\gamma}\mathcal{J}_{z,\gamma}\mathcal{L}^{\gamma}_n=\mathcal{L}^{\gamma}_n$
and
$\mathcal{J}_{z,\gamma}\mathcal{J}_{\bar{z},\gamma}\mathcal{L}^{\gamma}_n=\mathcal{L}^{\gamma}_n$,
where $\bar{z}=(\bar{z_1},\ldots,\bar{z_d})$ (it follows from
(\ref{eqlag}) and (\ref{tttss})). Hence,
$\mathcal{J}_{z,\gamma}:L^2(\RR^d_+,t^{\gamma}dt)\rightarrow
L^2(\RR^d_+,t^{\gamma}dt)$ is bijective with an inverse
$\mathcal{J}_{\bar{z},\gamma}$. Incidentally, we can also conclude
that $\mathcal{J}_{z,\gamma}:\SSS(\RR^d_+)\rightarrow
\SSS(\RR^d_+)$ is a topological isomorphism (has an inverse
$\mathcal{J}_{\bar{z},\gamma}$).\\
\indent Let $z\in\mathbf{T}^{(d)}$ and $\Phi_z(t)=\prod_{l=1}^d
e^{-\frac{1}{2}\frac{1+z_l}{1-z_l}t_l}$. Since $(1+z_l)/(1-z_l)$
is purely imaginary, for all $l=1,\ldots,d$, $|\Phi_z(t)|=1$ and
one easily verifies that the mapping $f\mapsto \Phi_z f$, is a
topological isomorphism
on $\SSS(\RR^d_+)$ and an isometry from $L^2(\RR^d_+,t^{\gamma}dt)$ onto itself. Since $\mathcal{I}_{z,\gamma}f=\Phi_z\mathcal{J}_{z,\gamma}(\Phi_zf)$ we can conclude that $\mathcal{I}_{z,\gamma}$ is topological isomorphism on $\SSS(\RR^d_+)$ and isometry from $L^2(\RR^d_+,t^{\gamma}dt)$ onto itself; clearly, its inverse is $\mathcal{I}_{\bar{z},\gamma}$.\\
\indent Now, by the same technique as in the proof of \cite[Lemma 3.2]{D2}, we have:\\
\begin{equation}\label{D2 3.4}
\left\|t^{(p+k+\gamma)/2}D^pf(t)\right\|_2=\left(\prod_{l=1}^d|1-z_l|^{-p_l+k_l}\right)
\left\|t^{(p+k+\gamma)/2}D^k\mathcal{J}_{z,\gamma}f(t)\right\|_2,\;f\in
\SSS(\mathbb{R}^d_+),
\end{equation}
\noindent for $\gamma\in\overline{\RR^d_+}$ and
$p,k\in\mathbb{N}_0^d$.

Next, we summarise the properties of $\mathcal{J}_{z,\gamma}$ and
$\mathcal{I}_{z,\gamma}$ in the following proposition:

\begin{prop}\label{HCt}
For $\gamma\in\overline{\RR^d_+}$ and $z\in\mathbf{T}^{(d)}$ the
fractional powers and the modified fractional powers of the
Hankel-Clifford transform $\mathcal{I}_{z,\gamma}$ and
$\mathcal{J}_{z,\gamma}$ are topological isomorphisms on
$\SSS(\RR^d_+)$ and they extend to isometries from
$L^2(\RR^d_+,t^{\gamma}dt)$ onto itself with inverses,
$\mathcal{I}_{\bar{z},\gamma}$ and $\mathcal{J}_{\bar{z},\gamma}$
respectively. Moreover, for all $p,k\in\NN^d_0$ and
$f\in\SSS(\RR^d_+)$, (\ref{D2 3.4}) is valid.
\end{prop}

Notice that when $z=-\mathbf{1}\in\mathbf{T}^{(d)}$ then
$\mathcal{H}_{\gamma}=\mathcal{J}_{z,\gamma}=\mathcal{I}_{z,\gamma}$
where $\mathcal{H}_{\gamma}$ is the $d$-dimensional
Hankel-Clifford transform, defined as \beas
\mathcal{H}_{\gamma}(f)(t)=2^{-d}t^{-\gamma/2}\int_{\mathbb{R}_+^d}
f(x)x^{\gamma/2}\prod_{l=1}^d
J_{\gamma_l}(\sqrt{x_lt_l})dx,\;t\in\mathbb{R}^d_+. \eeas By
(\ref{tttss}), $\mathcal{L}^{\gamma}_n$, $n\in\NN^d_0$, are
eigenfunctions for $\mathcal{H}_{\gamma}$; more precisely
$\mathcal{H}_{\gamma}\mathcal{L}^{\gamma}_n=(-1)^{|n|}\mathcal{L}^{\gamma}_n$.\\
\indent Since $\mathcal{J}_{z,0}$ is an isomorphism on
$\SSS(\mathbb{R}^d_+)$, by (\ref{D2 3.4}) we have the following
result.

\begin{thm}\label{Theorem A}
The modified fractional powers of the Hankel-Clifford transform
$\mathcal{J}_{z,0}$ are isomorphisms of
$G_\alpha(\mathbb{R}^d_+)$, $G^\beta(\mathbb{R}^d_+)$ and
$G_\alpha^\beta(\mathbb{R}^d_+)$ onto $G^\alpha(\mathbb{R}^d_+)$,
$G_\beta(\mathbb{R}^d_+)$ and $G^\alpha_\beta(\mathbb{R}^d_+)$
respectively.
\end{thm}

Proposition \ref{HCt} is also valid for the modified fractional
powers of the partial Hankel-Clifford transform. To make this
precise let $d',d''\in\NN$,
$\gamma=(\gamma',\gamma'')\in\overline{\RR^{d'}_+}\times
\overline{\RR^{d''}_+}=\overline{\RR^d_+}$ (for brevity
$d=d'+d''$) and $z'=(z_1,\ldots,z_{d'})\in\mathbf{T}^{(d')}$.
Denote by $\mathcal{J}^{d'}_{z',\gamma'}$ the modified fractional
power of the Hankel-Clifford transform on $\RR^{d'}_+$ and by
$\mathrm{Id}^{d''}$ the identity operator
$\SSS(\RR^{d''}_+)\rightarrow \SSS(\RR^{d''}_+)$. Theorem
\ref{kerth} and Proposition \ref{HCt} imply that
$\mathcal{J}^{d'}_{z',\gamma'}\hat{\otimes}\mathrm{Id}^{d''}$ is a
topological isomorphism on $\SSS(\RR^d_+)$ (it follows that
$\mathcal{J}^{d'}_{z',\gamma'}\hat{\otimes}\mathrm{Id}^{d''}$ is
an injection from \cite[Theorem 5, p. 277]{kothe2} and a
homomorphism from \cite[Theorem 7, p. 189]{kothe2}; note
$\SSS(\RR^d_+)$ is nuclear). We denote by $x\in\RR^d_+$
$x=(x',x'')$ where $x'=(x_1,\ldots,x_{d'})$ and
$x''=(x_{d'+1},\ldots,x_{d})$. Let $f\in\SSS(\RR^d_+)$. Define
\beas \mathcal{J}^{(d')}_{z',\gamma'}f(t)=
\left(\prod_{l=1}^{d'}(1-z_l)^{-1}\right)\int_{\RR^{d'}_+}f(x',t'')\prod_{l=1}^{d'}
(x_lt_lz_l)^{-\gamma_l/2}x_l^{\gamma_l}I_{\gamma_l}
\left(\frac{2\sqrt{x_lt_lz_l}}{1-z_l}\right)dx'. \eeas By the same
technique already described for the absolute convergence of
$\mathcal{J}_{z,\gamma}$, one proves that
$\mathcal{J}^{(d')}_{z',\gamma'}
f\in\mathcal{C}_{L^{\infty}}(\overline{\RR^d_+})$. When
$f_j\rightarrow f$ in $\SSS(\RR^d_+)$,
$\mathcal{J}^{(d')}_{z',\gamma'}f_j\rightarrow
\mathcal{J}^{(d')}_{z',\gamma'}f$ in
$\mathcal{C}_{L^{\infty}}(\overline{\RR^d_+})$. Since
$\mathcal{J}^{(d')}_{z',\gamma'}f(t)=\mathcal{J}^{d'}_{z',\gamma'}\hat{\otimes}
\mathrm{Id}^{d''}f(t)$ for
$f\in\SSS(\RR^{d'}_+)\otimes\SSS(\RR^{d''}_+)$, we accomplish the
same for all $f\in\SSS(\RR^d_+)$. Hence, the first part of the
next proposition follows.

\begin{prop}\label{parHC}
The modified fractional power of the partial Hankel-Clifford transform $\mathcal{J}^{(d')}_{z',\gamma'}$ is a topological isomorphism on $\SSS(\RR^d_+)$.\\
\indent Moreover, $\mathcal{J}^{(d')}_{z',\gamma'}$ extends to an
isometry from $L^2(\RR^d_+,t^{\gamma}dt)$ onto itself with ab
inverse $\mathcal{J}^{(d')}_{\bar{z'},\gamma'}$.
For all $(p',p''),(k',k'')\in\NN^{d'}_0\times\NN^{d''}_0=\NN^d_0$ and all $f\in\SSS(\RR^d_+)$\\
\\
$\left\|t'^{(p'+k'+\gamma')/2}t''^{(p''+k'')/2}D^{p}_{t}f(t)\right\|_2
$ \beas =\left(\prod_{l=1}^{d'}|1-z_l|^{-p_l+k_l}\right)
\left\|t'^{(p'+k'+\gamma')/2}t''^{(p''+k'')/2}D^{k'}_{t'}D^{p''}_{t''}
\mathcal{J}^{(d')}_{z',\gamma'}f(t)\right\|_2. \eeas
\end{prop}

\begin{proof} The proof that $\mathcal{J}^{(d')}_{z',\gamma'}$ extends to an isometry from $L^2(\RR^d_+,t^{\gamma}dt)$ onto itself with an inverse
$\mathcal{J}^{(d')}_{\bar{z'},\gamma'}$ is the same as for
$\mathcal{J}_{z,\gamma}$
given above. As in the proof of \cite[Lemma 3.2 $iii)$]{D2}, one obtains for $f\in\SSS(\RR^d_+)$\\
\\
$\left\|t'^{(p'+k'+\gamma')/2}t''^{(p''+k'')/2}D^{p'}_{t'}f(t)\right\|_2$
\bea\label{eqinttt}
=\left(\prod_{l=1}^{d'}|1-z_l|^{-p_l+k_l}\right)
\left\|t'^{(p'+k'+\gamma')/2}t''^{(p''+k'')/2}D^{k'}_{t'}
\mathcal{J}^{(d')}_{z',\gamma'}f(t)\right\|_2. \eea

\noindent Clearly, for
$f\in\SSS(\RR^{d'}_+)\otimes\SSS(\RR^{d''}_+)$, $D^{p''}_{t''}
\mathcal{J}^{(d')}_{z',\gamma'}f=\mathcal{J}^{(d')}_{z',\gamma'}D^{p''}_{t''}
f$. Hence, the same holds for $f\in\SSS(\RR^d_+)$ and the equality
follows from (\ref{eqinttt}). \qed
\end{proof}

If $\Lambda'=\{\lambda'_1,\ldots,\lambda'_{d'}\}\subseteq
\{1,\ldots,d\}$ and
$\Lambda''=\{\lambda''_1,\ldots,\lambda''_{d''}\}=\{1,\ldots,d\}\backslash
\Lambda'$ one can also consider the modified fractional power of
the partial Hankel-Clifford transform with respect to
$x_{\Lambda'}=(x_{\lambda'_1},\ldots,x_{\lambda'_{d'}})$ defined
by (here
$x_{\Lambda''}=(x_{\lambda''_1},\ldots,x_{\lambda''_{d''}})$ and
abusing the notation we write $x=(x_{\Lambda'},x_{\Lambda''})$)
$$\mathcal{J}^{(\Lambda')}_{z',\gamma_{\Lambda'}}f(t)=
\left(\prod_{l=1}^{d'}(1-z_l)^{-1}\right)
\int_{\RR^{d'}_+}f(x_{\Lambda'},t_{\Lambda''}) \prod_{l=1}^{d'}
(x_{\lambda'_l}t_{\lambda'_l}z_l)^{-\gamma_{\lambda'_l}/2}
x_{\lambda'_l}^{\gamma_{\lambda'_l}}I_{\gamma_{\lambda'_l}}
\left(\frac{2\sqrt{x_{\lambda'_l}t_{\lambda'_l}z_l}}{1-z_l}\right)
dx_{\Lambda'}.$$

\begin{cor}\label{corforthepp}
Using the same notations as above,
$\mathcal{J}^{(\Lambda')}_{z',\gamma_{\Lambda'}}$ is a topological
isomorphism on $\SSS(\RR^d_+)$ and it extends to an isometry from
$L^2(\RR^d_+,t^{\gamma}dt)$ onto itself with an inverse
$\mathcal{J}^{(\Lambda')}_{\bar{z'},\gamma_{\Lambda'}}$. For all $f\in\SSS(\RR^d_+)$ and all $(p_{\Lambda'},p_{\Lambda''}),(k_{\Lambda'},k_{\Lambda''})\in\NN^d_0$\\
\\
$\left\|t_{\Lambda'}^{(p_{\Lambda'}+k_{\Lambda'}+\gamma_{\Lambda'})/2}
t_{\Lambda''}^{(p_{\Lambda''}+k_{\Lambda''})/2}D^p_tf(t)\right\|_2
$ \bea\label{forrandomin}
=\left(\prod_{l=1}^{d'}|1-z_l|^{-p_{\lambda'_l}+k_{\lambda'_l}}\right)
\left\|t_{\Lambda'}^{(p_{\Lambda'}+k_{\Lambda'}+\gamma_{\Lambda'})/2}
t_{\Lambda''}^{(p_{\Lambda''}+k_{\Lambda''})/2}
D^{k_{\Lambda'}}_{t_{\Lambda'}} D^{p_{\Lambda''}}_{t_{\Lambda''}}
\mathcal{J}^{(\Lambda')}_{z',\gamma_{\Lambda'}}f(t)\right\|_2.
\eea
\end{cor}

\begin{proof} Let $\Theta:\RR^d\rightarrow\RR^d$ be the orthogonal transformation given by $\Theta(x)=y$, where
$y_{\lambda'_1}=x_1,\ldots,y_{\lambda'_{d'}}=x_{d'},
y_{\lambda''_1}=x_{d'+1},\ldots,y_{\lambda''_{d''}}=x_d$. Observe
that $\Theta$ maps $\RR^d_+$ and $\overline{\RR^d_+}$ bijectively
onto themselves. Let $\tilde{\Theta}$ be the mapping $f\mapsto
f\circ\Theta$, $L^2(\RR^d_+)\rightarrow L^2(\RR^d_+)$. One easily
verifies that for each $\mu\in\overline{\RR^d_+}$ it is an
isometry from $L^2(\RR^d_+,t^{\mu}dt)$ onto
$L^2(\RR^d_+,t^{\Theta^{-1}\mu}dt)$ and a topological isomorphism
on $\SSS(\RR^d_+)$. Its inverse is
$\tilde{\Theta}^{-1}f=f\circ\Theta^{-1}$. Let
$\nu'=(\gamma_{\lambda'_1},\ldots,\gamma_{\lambda'_{d'}})\in\overline{\RR^{d'}_+}$.
The corollary follows from Proposition \ref{parHC} and the fact
that $\mathcal{J}^{(\Lambda')}_{z',\gamma_{\Lambda'}}f
=\tilde{\Theta}^{-1}\mathcal{J}^{(d')}_{z',\nu'} \tilde{\Theta}f$.
\qed
\end{proof}

\begin{rem}\label{partHCrem}
Observe that if $\Lambda'=\emptyset$ then
$\mathcal{J}^{(\Lambda')}_{z',\gamma_{\Lambda'}}=\mathrm{Id}$ and
when
$\Lambda'=\{1,\ldots,d\}$, $\mathcal{J}^{(\Lambda')}_{z',\gamma_{\Lambda'}}$ is just $\mathcal{J}_{z,\gamma}$.\\
\indent Let $z'=-\mathbf{1}\in\mathbf{T}^{(d')}$ in
$\mathcal{J}^{(\Lambda')}_{z',\gamma_{\Lambda'}}$ we obtain the
partial Hankel-Clifford transform with respect to
$x_{\Lambda'}=(x_{\lambda'_1},\ldots,x_{\lambda'_{d'}})$ denoted
by $\mathcal{H}^{(\Lambda')}_{\gamma_{\Lambda'}}$.
\end{rem}

As a direct consequence of Corollary \ref{corforthepp} we have the
following result.

\begin{cor}\label{lllkk}
$\mathcal{J}^{(\Lambda')}_{z',0}$ is a topological isomorphism on
$G^{\alpha}_{\alpha}(\RR^d_+)$ with an inverse
$\mathcal{J}^{(\Lambda')}_{\bar{z'},0}$. In particular,
$\mathcal{H}^{(\Lambda')}_{0}$ is a self-inverse topological
isomorphism on $G^{\alpha}_{\alpha}(\RR^d_+)$.
\end{cor}

 \section{Fourier-Laguerre coefficients in
 $G_\alpha^\alpha(\mathbb{R}^d_+)$, $\alpha\geq 1$}

In this section, we characterise the space
$G_\alpha^\alpha(\mathbb{R}^d_+)$, $\alpha\geq 1$ in terms of the
Fourier-Laguerre coefficients.

\begin{prop}(\cite[Lemma 3.1]{D1}, for d=1)\label{D Lemma3.1}
Let $f\in L^2(\mathbb{R}^d_+)$ and $a_n=\int_{\mathbb{R}^d_+}f(t)
\mathcal{L}_{n}(t)dt$, $ n\in\mathbb{N}^d_0.$ If there exist
constants $c>0$ and $a>1$ such that

\begin{equation}\label{D 3.1}
|a_n|\leq ca^{-|n|^{1/\alpha}},\;n\in\mathbb{N}_0^d,
\end{equation}

\noindent then $f\in G_\alpha^\alpha(\mathbb{R}^d_+)$, $\alpha\geq
1$.
\end{prop}

\begin{proof} As $\{a_n\}_{n\in\mathbb{N}_0^d}\in \sss^{\alpha}\subseteq \sss$, it
follows $f\in \SSS(\mathbb{R}^d_+)$ and the series $\sum_n
a_n\mathcal{L}_n$ converges absolutely in $\SSS(\RR^d_+)$ to $f$.
Since $n_1^{1/\alpha}+\ldots+n_d^{1/\alpha}\leq d|n|^{1/\alpha}$,
denoting $\tilde{a}=a^{1/d}>1$, we have
$a^{-|n|^{1/\alpha}}\leq \prod_{l=1}^d \tilde{a}^{-n_l^{1/\alpha}}$.\\
\indent Using the estimates (2.5) and (2.6) given in the proof of
\cite[Lemma 2.1]{D1}, for $p\in\NN^d_0$ we have \bea
\left\|t^{p/2}\mathcal{L}_n(t)\right\|_2&\leq&
2^{|p|+5d}\prod_{l=1}^d (n_l+1)\ldots\left(n_l+\left[\frac{p_l}{2}\right]+2\right),\label{estforlaged}\\
\left\|t^{p/2}D^p\mathcal{L}_n(t)\right\|_2&\leq&
2^{5d}\prod_{l=1}^d
(n_l+1)\ldots\left(n_l+\left[\frac{p_l}{2}\right]+2\right)\label{estforlagdv}
\eea. Let $\Lambda=\{\lambda_1,\ldots,\lambda_{d'}\}\subseteq
\{1,\ldots,d\}$. Since
$\mathcal{H}^{(\Lambda)}_0\mathcal{L}_n=(-1)^{n_{\lambda_1}+\ldots+n_{\lambda_{d'}}}
\mathcal{L}_n$, (\ref{estforlaged}) implies \beas
\left\|t^{p/2}\mathcal{H}^{(\Lambda)}_0f(t)\right\|_2 &\leq&
\sum_{n\in\mathbb{N}_0^d}
|a_n|\left\|t^{p/2}\mathcal{L}_n(t)\right\|_2\\
&\leq& c2^{|p|+5d}\sum_{n\in\NN^d_0}\prod_{l=1}^d\tilde{a}^{-n_l^{1/\alpha}} (n_l+1)\ldots\left(n_l+\left[\frac{p_l}{2}\right]+2\right)\\
&\leq&c2^{|p|+5d}\prod_{l=1}^d\tilde{a}^{([p_l/2]+2)}\sum_{n\in\mathbb{N}_0^d}\prod_{l=1}^d
\tilde{a}^{-(n_l+[p_l/2]+2)^{1/\alpha}}
\left(n_l+\left[\frac{p_l}{2}\right]+2\right)^{[p_l/2]+2}. \eeas
Let $u>0$, $v>1$. Clearly,
$\rho_{u,v}(x)=v^{-(x+u)^{1/\alpha}}(x+u)^u$, $x\in(-u,+\infty)$
attains its maximum at $x=(\alpha u/\ln v)^{\alpha}-u$. This
implies that there exist $C_1,A_1>0$ such that
\bea\label{estforgro}
\left\|t^{p/2}\mathcal{H}^{(\Lambda)}_0f(t)\right\|_2\leq
C_1A_1^{|p|}p^{(\alpha/2)p},\,\,\, \mbox{for all}\,\,\,
p\in\mathbb{N}_0^d,\,\, \Lambda\subseteq \{1,\ldots,d\}. \eea
Similarly, by using (\ref{estforlagdv}), there exist $C_2,A_2>0$
such that \bea\label{estforder}
\left\|t^{p/2}D^p\mathcal{H}^{(\Lambda)}_0f(t)\right\|_2\leq
C_2A_2^{|p|}p^{(\alpha/2)p},\,\,\, \mbox{for all}\,\,\,
p\in\mathbb{N}_0^d,\,\, \Lambda\subseteq \{1,\ldots,d\}. \eea
Denote by $(\cdot,\cdot)$ the inner product in $L^2(\RR^d_+)$.
Since $\mathcal{H}^{(\Lambda)}_0f\in\SSS(\RR^d_+)$, by integration
by parts one easily verifies that \beas \left(t^{(p+k)/2}
D^p\mathcal{H}^{(\Lambda)}_0f(t),t^{(p+k)/2}
D^p\mathcal{H}^{(\Lambda)}_0f(t)\right)=
\left|\left(D^p\left(t^{p+k}
D^p\mathcal{H}^{(\Lambda)}_0f(t)\right),\mathcal{H}^{(\Lambda)}_0f(t)\right)\right|.
\eeas
Hence, for all $k,p\in\NN^d_0$ such that $2k\geq p$ by (\ref{estforgro}) and (\ref{estforder}), we obtain\\
\\
$\ds \left\|t^{(p+k)/2}
D^p\mathcal{H}^{(\Lambda)}_0f(t)\right\|^2_2$ \beas &\leq&
\sum_{m\leq p} \binom{p}{m}\frac{(p+k)!}{(p+k-m)!}
\left|\left(t^{p+k-m}D^{2p-m}\mathcal{H}^{(\Lambda)}_0f(t),\mathcal{H}^{(\Lambda)}_0f(t)\right)\right|\\
&\leq& 2^{|p|+|k|}\sum_{m\leq p} \binom{p}{m}m!
\left|\left(t^{(2p-m)/2}D^{2p-m}\mathcal{H}^{(\Lambda)}_0f(t),t^{(2k-m)/2}\mathcal{H}^{(\Lambda)}_0f(t)\right)\right|\\
&\leq& C'A'^{|p|+|k|}\sum_{m\leq p} \binom{p}{m}m^{(\alpha/2)
m}m^{(\alpha/2) m}(2p-m)^{(\alpha/2)(2p-m)}
(2k-m)^{(\alpha/2)(2k-m)}\\
&\leq& C'A'^{|p|+|k|}2^{|p|}(2p)^{\alpha p}(2k)^{\alpha k}, \eeas
i.e. there exist $C_3,A_3>0$ such that for all $k,p\in\NN^d_0$
such that $2k\geq p$ and all $\Lambda\subseteq \{1,\ldots,d\}$
\bea\label{eq111}
\left\|t^{(p+k)/2}D^p\mathcal{H}^{(\Lambda)}_0f(t)\right\|_2\leq
C_3 A_3^{|p+k|}p^{(\alpha/2)p}k^{(\alpha/2)k}. \eea Let now
$p,k\in\NN^d_0$ be arbitrary but fixed. Let
$\Lambda'=\{\lambda'_1,\ldots,\lambda'_{d'}\}\subseteq
\{1,\ldots,d\}$ be such that $k_{\lambda'_l}<p_{\lambda'_l}/2$,
$l=1,\ldots,d'$ and
$\Lambda''=\{\lambda''_1,\ldots,\lambda''_{d''}\}=\{1,\ldots,d\}\backslash
\Lambda'$ be such that $k_{\lambda''_l}\geq p_{\lambda''_l}/2$,
$l=1,\ldots,d''$. Then (\ref{forrandomin}) and (\ref{eq111}) imply
\beas \left\|t^{(p+k)/2}D^p_tf(t)\right\|_2\leq 2^{|k|}
\left\|t^{(p+k)/2} D^{k_{\Lambda'}}_{t_{\Lambda'}}
D^{p_{\Lambda''}}_{t_{\Lambda''}}
\mathcal{H}^{(\Lambda')}_0f(t)\right\|_2\leq
C_3(2A_3)^{|p+k|}p^{(\alpha/2)p}k^{(\alpha/2)k}, \eeas i.e. $f\in
G^{\alpha}_{\alpha}(\RR^d_+)$. \qed
\end{proof}

Our next goal is to prove that $f\in
G_\alpha^\alpha(\mathbb{R}^d_+)$
implies (\ref{D 3.1}). We need some preparations.\\
\indent Let $\mathbf{\Pi}=\Pi_1\times...\times\Pi_d$, where
$\Pi_l=\{z_l\in\mathbb{C}|\, \mathrm{Im}\,z_l<0\}$, $l=1,...,d$.
One easily verifies that for each $z=x+iy\in\mathbf{\Pi}$, the
function $t\mapsto e^{-2\pi i zt}$, $\RR^d_+\rightarrow \CC$, is
in $G^{\alpha}_{\alpha}(\RR^d_+)$ (also in $\SSS(\RR^d_+)$). Also,
for $z=x+iy\in\mathbf{\Pi}$, the functions $t\mapsto
D_{x_l}e^{-2\pi i(x+iy)t}=-2\pi i t_l e^{-2\pi i(x+iy)t}$,
$\RR^d_+\rightarrow \CC$ and $t\mapsto D_{y_l} e^{-2\pi
i(x+iy)t}=2\pi t_l e^{-2\pi i(x+iy)t}$, $\RR^d_+\rightarrow \CC$
are in $G^{\alpha}_{\alpha}(\RR^d_+)$ (also in $\SSS(\RR^d_+)$)
for $l=1,\ldots,d$. For the moment, denote by $e_l$,
$l=1,\ldots,d$, the point in $\RR^d$ which all coordinates are $0$
except the $l$-th coordinate which is equal to $1$. By standard
arguments, one proves that for the fixed
$x^{(0)}=(x^{(0)}_1,\ldots, x^{(0)}_d)\in \RR^d$ and
$y^{(0)}=(y^{(0)}_1,\ldots, y^{(0)}_d)\in \RR^d$ with
$y^{(0)}_l<0$, $l=1,\ldots,d$ (i.e. $z^{(0)}=x^{(0)}+iy^{(0)}\in
\mathbf{\Pi}$) we have \beas \left(e^{-2\pi i (x^{(0)}+x_l
e_l+iy^{(0)})t}-e^{-2\pi
i(x^{(0)}+iy^{(0)})t}\right)/x_l&\rightarrow&
-2\pi i t_l e^{-2\pi i(x^{(0)}+iy^{(0)})t},\,\, \mbox{as}\,\, x_l\rightarrow 0\,\, \mbox{and}\\
\left(e^{-2\pi i (x^{(0)}+i(y^{(0)}+y_l e_l))t}-e^{-2\pi
i(x^{(0)}+iy^{(0)})t}\right)/y_l&\rightarrow& 2\pi t_l e^{-2\pi
i(x^{(0)}+iy^{(0)})t},\,\, \mbox{as}\,\, y_l\rightarrow 0 \eeas in
$G^{\alpha,A}_{\alpha,A}(\RR^d_+)$ for some $A>0$ and consequently
in $G^{\alpha}_{\alpha}(\RR^d_+)$ and $\SSS(\RR^d_+)$. Moreover,
\beas -2\pi i t_l e^{-2\pi i(x+iy)t}\rightarrow -2\pi i t_l
e^{-2\pi i(x^{(0)}+iy^{(0)})t}\,\, \mbox{and}\,\, 2\pi t_l
e^{-2\pi i(x+iy)t}\rightarrow 2\pi t_l e^{-2\pi
i(x^{(0)}+iy^{(0)})t} \eeas as $(x,y)\rightarrow
(x^{(0)},y^{(0)})$ in $G^{\alpha,A}_{\alpha,A}(\RR^d_+)$ for some
$A>0$. Hence, the same holds in $G^{\alpha}_{\alpha}(\RR^d_+)$ and
$\SSS(\RR^d_+)$. It follows that for each
$u\in(G^{\alpha}_{\alpha}(\RR^d_+))'$ or $u\in (\SSS(\RR^d_+))'$,
the function $z\mapsto \mathcal{F}_{\mathbf{\Pi}}u(z)=\langle
u(t),e^{-2\pi izt}\rangle$, $\mathbf{\Pi}\rightarrow\CC$, is of
the class $\mathcal{C}^1$;
$D_{x_l}\mathcal{F}_{\mathbf{\Pi}}u(x+iy)= \langle
u(t),D_{x_l}e^{-2\pi i(x+iy)t}\rangle$ and
$D_{y_l}\mathcal{F}_{\mathbf{\Pi}}u(x+iy)= \langle
u(t),D_{y_l}e^{-2\pi i(x+iy)t}\rangle$. Since the Cauchy-Riemann
equations hold for $\mathcal{F}_{\mathbf{\Pi}}u$, it is analytic
on $\mathbf{\Pi}$. Let $\mathbf{D}=D_1\times...\times D_d$, where
$D_l=\{w_l\in\mathbb{C}|\,|w_l|<1\}$, $l=1,...,d$. Observe that
the mapping \beas w\mapsto \Omega(w)=\left((1+w_1)/(4\pi
i(1-w_1)), \ldots,(1+w_d)/(4\pi i(1-w_d))\right) \eeas is a
biholomorphic mapping from $\mathbf{D}$ onto $\mathbf{\Pi}$ with
an inverse \beas z\mapsto \Omega^{-1}(z)=\left((4\pi i
z_1-1)/(4\pi i z_1+1),\ldots, (4\pi i z_d-1)/(4\pi i
z_d+1)\right). \eeas Thus, we have the following result.

\begin{lem}\label{analyticitt}
For each $u\in (G^{\alpha}_{\alpha}(\RR^d_+))'$ or $u\in
(\SSS(\RR^d_+))'$, the function \beas
\mathcal{F}_{\mathbf{D}}u(w)=\mathcal{F}_{\mathbf{\Pi}}u(\Omega(w))=\left\langle
u(t),\prod_{l=1}^d
e^{-\frac{1}{2}\frac{1+w_l}{1-w_l}t_l}\right\rangle,\,\,
\mathbf{D}\rightarrow \CC, \eeas is analytic on $\mathbf{D}$, i.e.
$\mathcal{F}_{\mathbf{D}}u\in\mathcal{O}(\mathbf{D})$.
\end{lem}

\begin{prop}(\cite[Proposition 1.1]{D3}, for d=1)\label{Theorem C}
Let $u\in(\SSS(\mathbb{R}_+^d))'$ and $a_n=\langle
u,\mathcal{L}_n\rangle$, $n\in\mathbb{N}^d_0$. Then,
\bea\label{335tt}
\mathcal{F}_\mathbf{D}(u)(w)=\prod_{j=1}^d(1-w_j)\sum_{n\in\mathbb{N}^d_0}a_{n}w^{n},\;w\in
\mathbf{D}. \eea In particular, if $\mathcal{F}_{\mathbf{D}}u=0$
then $u=0$.
\end{prop}

\begin{proof}
By Theorem \ref{razvoj u S+'},
$u=\sum_{n\in\mathbb{N}^d_0}a_n\mathcal{L}_n$ and the series
converges absolutely in $(\SSS(\mathbb{R}^d_+))'$. As $e^{-2\pi
izt}\in \SSS(\mathbb{R}^d_+)$, $z\in\mathbf{\Pi}$, we obtain
$$\mathcal{F}_{\mathbf{\Pi}}(u)(z)=\sum_{n\in\mathbb{N}^d_0}a_n \int_{\mathbb{R}^d_+}\mathcal{L}_n(t)e^{-2\pi izt}dt,\,\,
z\in\mathbf{\Pi}.$$ Moreover, as (see \cite[p. 191]{Ed})
$$\int_0^\infty t^\gamma
L_n^\gamma(t)e^{-st}dt=\frac{\Gamma(n+1+\gamma)(s-1)^n}{n!s^{\gamma+n+1}},\,\,
\gamma>-1,\,\, \mathrm{Re}\,s>0,$$ \noindent we obtain
$$\mathcal{F}_{\mathbf{\Pi}}(u)(z)=\sum_{n\in\mathbb{N}^d_0}a_n\prod_{j=1}^d\frac{(\frac{1}{2}+2\pi
iz_j-1)^{n_j}}{(\frac{1}{2}+2\pi iz_j)^{n_j+1}},
\,\,z\in\mathbf{\Pi}.$$ By the definition of
$\mathcal{F}_{\mathbf{D}}u$, (\ref{335tt}) follows. \qed
\end{proof}

The next two assertions are already proved in \cite{D1}, Lemma 3.2
and Corollary 3.5, in the case d=1. However, there are subtle gaps
which we improve upon.

\begin{prop}\label{D Lemma3.2}
Let $\alpha\geq 1$ and $\{a_n\}_{n\in\NN^d_0}$ be a sequence of
complex numbers such that $a_n\rightarrow 0$ as $|n|\rightarrow
\infty$. Then
$$F(w)=(\mathbf{1}-w)^{\mathbf{1}}\sum_{n\in\mathbb{N}^d_0}a_nw^{n},\; w\in
\mathbf{D},$$ belongs to $\mathcal{O}(\mathbf D)$. The following
conditions are equivalent:
\begin{enumerate}
\item[$(i)$] There exist constants $C,A>0$ such that
\begin{equation}\label{D 3.5}
|D^pF(w)|\leq CA^{|p|}p^{\alpha p},\quad
p\in\mathbb{N}^d_0\;,w\in\mathbf{D}.
\end{equation}
\item[$(ii)$] There exist constants $c>0$, $a>1$ such that
$|a_n|\leq ca^{-|n|^{1/\alpha}},\;n\in\mathbb{N}^d_0$.
\end{enumerate}
\end{prop}

\begin{proof} Clearly $F\in \mathcal{O}(\mathbf{D})$.
Let $\sum_{n\in\NN^d_0} b_nw^n$ be the power series expansion of
$F$ at $0$. Then, for $n\in\NN^d_0$ we have \bea
b_n&=&\frac{D^nF(0)}{n!}=\frac{1}{n!}\sum_{\substack{k\leq n\\
k\leq
\mathbf{1}}}\binom{n}{k} (-1)^{|k|}\left(\sum_{m\geq n-k}\frac{m!}{(m-n+k)!} a_mw^{m-n+k}\right)\Bigg|_{w=0}\nonumber\\
&=& \sum_{\substack{k\leq n\\ k\leq
\mathbf{1}}}(-1)^{|k|}a_{n-k}.\label{estpb} \eea Thus, for
$n,m\in\NN^d_0$, \bea\label{sumft} \sum_{p\leq
m}b_{n+\mathbf{1}+p}=\sum_{p\leq m}\sum_{k\leq
\mathbf{1}}(-1)^{|k|}a_{n+p+\mathbf{1}-k}. \eea Firstly, assume
that $d\geq 2$. Denote by $Q_{n,m}$ the $d$-dimensional
parallelepiped $\{x\in\RR^d|\, n_l\leq x_l\leq n_l+m_l+1, \,
l=1,\ldots,d\}$. If $q\in\NN^d_0$ is such that $n+q$ is in the
interior of $Q_{n,m}$ then $a_{n+q}$ appears exactly $2^d$ times
in the sum on the right hand side of (\ref{sumft}) such that
$2^{d-1}$ times with the "$+$" sign and $2^{d-1}$ times with "$-$"
sign. If $n+q$ is on the $s$-dimensional face of $Q_{n,m}$, $1\leq
s\leq d-1$, then $a_{n+q}$ appears exactly $2^s$ times half of
which are with the "$+$" sign and the other half with the "$-$"
sign. Thus on the right hand side of (\ref{sumft}) everything
cancels except for those terms which indexes are the vertices of
$Q_{n,m}$ and they appear only once. For $k\in\NN^d_0$ with $k\leq
\mathbf{1}$ denote by $m^{(k)}$ the multi-index that satisfies
$m^{(k)}_l=0$ if $k_l=0$ and $m^{(k)}_l=m_l+1$ if $k_l=1$,
$l=1,\ldots,d$; when $k$ varies through the multi-indexes that are
$\leq \mathbf{1}$, $n+m+\mathbf{1}-m^{(k)}$ varies through the
vertices of $Q_{n,m}$. Using this notations, by the above
observations, we have \bea\label{ieqfc} \sum_{p\leq
m}b_{n+\mathbf{1}+p}=\sum_{k\leq
\mathbf{1}}(-1)^{|k|}a_{n+m+\mathbf{1}-m^{(k)}},\,\, \forall
n,m\in\NN^d_0. \eea
Clearly, for $d=1$ (\ref{sumft}) and (\ref{ieqfc}) are equal.\\
\indent Assume that $(i)$ holds. Since $D^pF(w)=\sum_{n\geq p}
\left(n!/(n-p)!\right)b_nw^{n-p}$, the hypothesis in $(i)$ and the
Cauchy formula yield $\left(n!/(n-p)!\right)|b_n|\leq
CA^{|p|}p^{\alpha p}$, for all $n,p\in\mathbb{N}^d_0$, $n\geq p$.
As $n!/(n-p)!\geq e^{-|p|}n^p$, for $n\geq p$, we have
\begin{equation}\label{D 3.6}
|b_n|\leq C\prod_{j=1}^d\inf_{p_j\leq n_j} \frac{(e
A)^{p_j}p_j^{\alpha p_j}}{n_j^{p_j}},\,\,n\in\mathbb{N}_0^d.
\end{equation}
Of course we can assume $A\geq 1$. Then, if $p_j\geq n_j$,
$$(e A)^{p_j}p_j^{\alpha p_j}/n_j^{p_j}\geq(e A)^{n_j}n_j^{\alpha n_j}/n_j^{n_j},$$
\noindent and so the infimum in (\ref{D 3.6}) can be taken varying
on $p_j\geq 0$, $j=1,\ldots,d$. Thus, \cite[(2) and (3), p.
169-170]{GS1} imply, with suitable $c'>0$ and $a'>1$,
$$|b_n|\leq C\prod_{j=1}^d\inf_{p_j\geq 0}\frac{(e A)^{p_j}p_j^{\alpha
p_j}}{n_j^{p_j}}\leq
c'a'^{-|n|^{1/\alpha}},\,\,n\in\mathbb{N}_0^d.$$ Observe that for
$p,n\in\NN^d_0$ with $p\geq n$, we have $|p|^{1/\alpha}\geq
(|p-n|^{1/\alpha}+|n|^{1/\alpha})/2$. Thus, if we put
$a=\sqrt{a'}>1$ we have $a'^{-|p|^{1/\alpha}}\leq
a^{-|p-n|^{1/\alpha}}a^{-|n|^{1/\alpha}}$ for all $p\geq n$. The
above estimate for $|b_n|$ together with (\ref{ieqfc}) implies
that for all $n,m\in\NN^d_0$ \beas
|a_n|&\leq& \sum_{p\leq m}|b_{n+\mathbf{1}+p}|+\sum_{\substack{k\leq \mathbf{1}\\k\neq \mathbf{1}}}|a_{n+m+\mathbf{1}-m^{(k)}}|\leq c'a^{-|n|^{1/\alpha}}\sum_{p\in\NN^d_0} a^{-|p|^{1/\alpha}}+\sum_{\substack{k\leq \mathbf{1}\\k\neq \mathbf{1}}}|a_{n+m+\mathbf{1}-m^{(k)}}|\\
&=& ca^{-|n|^{1/\alpha}}+\sum_{\substack{k\leq \mathbf{1}\\k\neq
\mathbf{1}}}|a_{n+m+\mathbf{1}-m^{(k)}}|. \eeas
The last sum has exactly $2^d-1$ terms and since $k\neq \mathbf{1}$, $|n+m+\mathbf{1}-m^{(k)}|\geq |n|+\min\{m_l|\, l=1,\ldots ,d\}$. Let $n\in\NN^d_0$ be arbitrary but fixed. Since the above estimate for $|a_n|$ holds for arbitrary $m\in\NN^d_0$ and since $a_n\rightarrow 0$ as $|n|\rightarrow\infty$ (by hypothesis), this implies $|a_n|\leq ca^{-|n|^{1/\alpha}}$.\\
\indent Assume now that $(ii)$ holds. Then (\ref{estpb}) implies
the existence of $a>1$ and $c>0$ such that $|b_n|\leq c
a^{-|n|^{1/\alpha}}$, $\forall n\in\NN^d_0$. Observe that
$n_1^{1/\alpha}+\ldots+n_d^{1/\alpha}\leq d|n|^{1/\alpha}$. Hence,
by putting $a'=a^{1/d}$, we have $a^{-|n|^{1/\alpha}}\leq
\prod_{j=1}^d a'^{-n_j^{1/\alpha}}$. Now, for $p\in\NN^d_0$ and
$w\in\mathbf{D}$ we obtain
$$|D^pF(w)|\leq \sum_{n\geq p}\frac{n!}{(n-p)!}|b_n|\leq c\sum_{n\in\mathbb{N}^d_0}\prod_{j=1}^d n_j^{p_j}a'^{-n_j^{1/\alpha}}.$$
Since $\rho(x)=x^pu^{-x^{1/\alpha}}$, $x\geq0$ ($u>1$,
$p\in\NN_0$) attains its maximum  at $x=(\alpha p/\ln
u)^{\alpha}$, we proved (\ref{D 3.5}). \qed
\end{proof}

We will prove in Proposition \ref{D Cor3.5} that for $f\in
G_\alpha^\alpha(\mathbb{R}^d_+)$, the analytic function
$\mathcal{F}_\mathbf{D}(f)$ satisfies part $(i)$ of the previous
proposition. In order to prove this we need the next result; its
proof is analogous to the proof of \cite[Theorem 3.3]{D1} for the
one dimensional case and we omit it.

\begin{prop}\label{D Th.3.3}
Let $f\in G_\alpha(\mathbb{R}^d_+)$, $\alpha\geq 1$. Then there
exist constants $C, A>0$ such that
$$|D^p\mathcal{F}_{\mathbf{D}}(f)(w)|\leq CA^{|p|}p^{\alpha p},\;p\in\mathbb{N}_0^d,\;w\in\mathbf{D},\;\mathrm{Re}\, w_l\leq
0,\,\, l=1,...,d.$$
\end{prop}

\begin{prop}\label{D Cor3.5}
Let $f\in G^{\alpha}_{\alpha}(\mathbb{R}^d_+)$, $\alpha\geq 1$.
Then there exist constants $C, A>0$ such that
\begin{equation}\label{D Cor3.5 (1)}
|D^p\mathcal{F}_{\mathbf{D}}(f)(w)|\leq CA^{|p|}p^{\alpha p},\;
p\in\mathbb{N}_0^d,\; w\in\mathbf{D}
\end{equation}
and $\lim_{w\rightarrow
\mathbf{1}}\mathcal{F}_{\mathbf{D}}(f)(w)=0$.
\end{prop}

\begin{proof}
As $f\in \SSS(\mathbb{R}^d_+)$, Proposition \ref{Theorem C}
implies
that $\lim_{w\rightarrow \mathbf{1}}\mathcal{F}_{\mathbf{D}}(f)(w)=0$.\\
\indent We introduce some notation to make the simpler simpler.
Let $\Lambda'=\{\lambda'_1,\ldots,\lambda'_{d'}\}\subseteq
\{1,\ldots,d\}$ and $\Lambda''=\{\lambda''_1,
\ldots,\lambda''_{d''}\}=\{1,\ldots,d\}\backslash\Lambda'$. For
$\zeta\in\CC^d$ (or in $\RR^d_+$, or in $\NN^d_0$) we denote
$\zeta_{\Lambda'}=(\zeta_{\lambda'_1},\ldots,
\zeta_{\lambda'_{d'}})$,
$\zeta_{\Lambda''}=(\zeta_{\lambda''_1},\ldots,\zeta_{\lambda''_{d''}})$
and by abusing the notation we write
$\zeta=(\zeta_{\Lambda'},\zeta_{\Lambda''})$. Let
$\tilde{\Lambda'}$ be the biholomorphic mapping from $\CC^d$ onto
itself defined by $\tilde{\Lambda'}w=\zeta$ where
$\zeta_{\lambda'_l}=-w_{\lambda'_l}$, $l=1,\ldots,d'$ and
$\zeta_{\lambda''_s}=w_{\lambda''_s}$, $s=1,\ldots,d''$. Also,
denote \beas \mathbf{D}_{(\Lambda')}=\{\zeta\in\mathbf{D}|\,
\mathrm{Re}\, \zeta_{\lambda'_l}\geq 0,\, l=1,\ldots,d',\,\,
\mbox{and}\,\, \mathrm{Re}\, \zeta_{\lambda''_s}\leq 0,\,
s=1,\ldots,d''\} \eeas (note that $\mathbf{D}_{(\emptyset)}$
consists of all $w\in\mathbf{D}$ such that the coordinates of $w$
have non-positive real
parts).\\
\indent For $f\in\SSS(\RR^d_+)$ let $a_n=\langle
f,\mathcal{L}_n\rangle$, $n\in\NN^d_0$. Then, Proposition
\ref{Theorem C} implies
$\mathcal{F}_{\mathbf{D}}f(w)=(\mathbf{1}-w)^{\mathbf{1}}\sum_{n\in\NN^d_0}
a_n w^n$, $w\in\mathbf{D}$. As $\langle
\mathcal{H}^{(\Lambda')}_0f,\mathcal{L}_n\rangle=
(-1)^{n_{\lambda'_1}+\ldots+ n_{\lambda'_{d'}}}a_n$ and \beas
\mathcal{F}_{\mathbf{D}}(\mathcal{H}^{(\Lambda')}_0f)(w)=(\mathbf{1}-w)^{\mathbf{1}}\sum_{n\in\NN^d_0}
(-1)^{n_{\lambda'_1}+ \ldots+ n_{\lambda'_{d'}}}a_n w^n,\,\,\,
w\in\mathbf{D}. \eeas Hence, \beas
\mathcal{F}_{\mathbf{D}}(\mathcal{H}^{(\Lambda')}_0f)(\tilde{\Lambda'}w)=\left(\prod_{l=1}^{d'}
\frac{1+w_{\lambda'_l}}{1-w_{\lambda'_l}}\right)\cdot(\mathbf{1}-w)^{\mathbf{1}}\sum_{n\in\NN^d_0}a_nw_n,\,\,\,
w\in\mathbf{D}. \eeas Thus \bea\label{est11kk}
\mathcal{F}_{\mathbf{D}}f(w)=\left(\prod_{l=1}^{d'}\frac{1-w_{\lambda'_l}}{1+w_{\lambda'_l}}\right)\mathcal{F}_{\mathbf{D}}
(\mathcal{H}^{(\Lambda')}_0f)(\tilde{\Lambda'}w),\,\,
w\in\mathbf{D},\, f\in\SSS(\RR^d_+). \eea Let $f\in
G^{\alpha}_{\alpha}(\RR^d_+)$. Since
$\mathcal{H}^{(\Lambda')}_0f\in G^{\alpha}_{\alpha}(\RR^d_+)$ (cf.
Corollary \ref{lllkk}), Proposition \ref{D Th.3.3} implies the
existence of $A,C>0$ such that \bea\label{est33kk}
\left|D^n\mathcal{F}_{\mathbf{D}}(\mathcal{H}^{(\Lambda')}_0f)(w)\right|\leq
CA^{|n|}n^{\alpha n},\,\, \forall n\in\NN^d_0,\, \forall w\in
\mathbf{D}_{(\emptyset)},\, \forall
\Lambda'\subseteq\{1,\ldots,d\}. \eea Observe that for
$w\in\mathbf{D}_{(\emptyset)}$, (\ref{D Cor3.5 (1)}) holds by
Proposition \ref{D Th.3.3}. To prove (\ref{D Cor3.5 (1)}) for
$w\in\mathbf{D}_{(\Lambda')}$ when $\emptyset\neq
\Lambda'\subseteq\{1,\ldots,d\}$, we need an estimate for the
derivatives of the function $\zeta\mapsto (1-\zeta)/(1+\zeta)$,
$\{\zeta\in\CC|\, |\zeta|<1\}\rightarrow \CC$ when $\mathrm{Re}\,
\zeta\geq 0$. Since $(1-\zeta)/(1+\zeta)=2/(1+\zeta)-1$ and
$|1+\zeta|\geq 1$ when $\mathrm{Re}\, \zeta\geq 0$, for $j\in\NN$
we have \bea\label{derestsup}
\left|\frac{d^j}{d\zeta^j}\left(\frac{1-\zeta}{1+\zeta}\right)\right|=\frac{2j!}{|1+\zeta|^{j+1}}\leq
2j!,\,\,\, \mbox{when}\,\,\, |\zeta|<1\,\, \mbox{and}\,\,
\mathrm{Re}\,\zeta\geq 0. \eea Clearly, (\ref{derestsup}) also
holds for $j=0$. Now, observe that
$\tilde{\Lambda'}(\mathbf{D}_{(\Lambda')})=\mathbf{D}_{(\emptyset)}$.
Hence, for $w\in\mathbf{D}_{(\Lambda')}$, (\ref{est11kk}),
(\ref{est33kk}) and (\ref{derestsup}) imply \beas
\left|D^n\mathcal{F}_{\mathbf{D}}f(w)\right|&\leq&
\sum_{m_{\Lambda'}\leq n_{\Lambda'}}
\binom{n_{\Lambda'}}{m_{\Lambda'}}
2^{d'}m_{\Lambda'}!\left|D^{n_{\Lambda'}-m_{\Lambda'}}_{w_{\Lambda'}}D^{n_{\Lambda''}}_{w_{\Lambda''}}
\mathcal{F}_{\mathbf{D}}(\mathcal{H}^{(\Lambda')}_0f)(\tilde{\Lambda'}w)\right|\\
&\leq& C_1\sum_{m_{\Lambda'}\leq n_{\Lambda'}}
\binom{n_{\Lambda'}}{m_{\Lambda'}} m_{\Lambda'}^{\alpha
m_{\Lambda'}}A^{|n|-|m_{\Lambda'}|}
(n_{\Lambda'}-m_{\Lambda'})^{\alpha (n_{\Lambda'}-m_{\Lambda'})}
n_{\Lambda''}^{\alpha n_{\Lambda''}}\\
&\leq& C_1(2A)^{|n|}n^{\alpha n}, \eeas which completes the proof.
\qed
\end{proof}

Now, Proposition \ref{D Lemma3.1}, Proposition \ref{D Cor3.5},
Proposition \ref{Theorem C} and Proposition \ref{D Lemma3.2} give
the main result of this section:

\begin{thm}\label{D Th3.6}(\cite[Theorem 3.6]{D1}, for d=1)
Let $\alpha\geq 1$. For $f\in L^2(\mathbb{R}^d_+)$ let
$a_n=\int_{\mathbb{R}^d_+}f(t)\mathcal{L}_{n}(t)dt$,
$n\in\mathbb{N}^d_0$. The following conditions are equivalent:
\begin{enumerate}
\item[$(i)$] There exist $c>0$ and $a>1$ such that $|a_n|\leq ca^{-|n|^{1/\alpha}}$ for $n\in\mathbb{N}^d_0$.
\item[$(ii)$] $f\in G_\alpha^\alpha(\mathbb{R}^d_+)$.
\item[$(iii)$] There exist $C, A>0$ such that
\beas |D^p\mathcal{F}_{\mathbf{D}}(f)(w)|\leq CA^{|p|}p^{\alpha
p}\quad\mbox{for}\;p\in\mathbb{N}^d_0\quad\mbox{and}\;\;
w\in\mathbf{D} \eeas and $\lim_{w\rightarrow
\mathbf{1}}\mathcal{F}_{\mathbf{D}}(f)(w)=0$.
\end{enumerate}
Conversely, given a sequence $\{a_n\}_{n\in\mathbb{N}_0^d}$
satisfying condition $(i)$ or given $F\in \mathcal{O}(\mathbf{D})$
of the form $F(w)=(\mathbf{1}-w)^{\mathbf{1}}\sum_n a_n w^n$ with
$a_n\rightarrow 0$ as $|n|\rightarrow \infty$ which satisfies
(\ref{D 3.5}), there exists $f\in G_\alpha^\alpha(\mathbb{R}^d_+)$
such that $a_n=\int_{\mathbb{R}^d_+}f(t)\mathcal{L}_n(t)dt$ and
$\mathcal{F}_{\mathbf{D}}(f)(w)=F(w)$ for $w\in\mathbf{D}$.
\end{thm}

\section{Topological properties of $G^{\alpha}_{\alpha}(\RR^d_+)$, $\alpha\geq 1$. The Kernel theorems}

As we shell see, we gain deep insights into the topological
structure of $G^{\alpha}_{\alpha}(\RR^d_+)$, $\alpha\geq 1$ by
Theorem \ref{D Th3.6}. Let $\iota:
G^{\alpha}_{\alpha}(\RR^d_+)\rightarrow \sss^{\alpha}$,
$\iota(f)=\{\langle f,\mathcal{L}_n\rangle\}_{n\in\NN^d_0}$.
Theorem \ref{D Th3.6} proves that $\iota$ is a well defined
bijection.

\begin{thm}\label{isoss}
Let $\alpha\geq 1$. The mapping $\iota:
G^{\alpha}_{\alpha}(\RR^d_+)\rightarrow \sss^{\alpha}$,
$\iota(f)=\{\langle f,\mathcal{L}_n\rangle\}_{n\in\NN^d_0}$, is a
topological isomorphism between $G^{\alpha}_{\alpha}(\RR^d_+)$ and
$\sss^{\alpha}$. In particular, $G^{\alpha}_{\alpha}(\RR^d_+)$ is
a $(DFN)$-space and $(G^{\alpha}_{\alpha}(\RR^d_+))'$ is
an $(FN)$-space.\\
\indent For each $f\in G^{\alpha}_{\alpha}(\RR^d_+)$,
$\sum_{n\in\NN^d_0}\langle f,\mathcal{L}_n\rangle \mathcal{L}_n$
is summable to $f$ in $G^{\alpha}_{\alpha}(\RR^d_+)$.
\end{thm}

\begin{proof} If we consider $\iota$ as a linear mapping from $G^{\alpha}_{\alpha}(\RR^d_+)$ into $\sss$ ($\sss^{\alpha}$ is
canonically injected into $\sss$) then $\iota$ is continuous since
it decomposes as $\ds G^{\alpha}_{\alpha}(\RR^d_+)\longrightarrow
\SSS(\RR^d_+)\xrightarrow{f\mapsto \{\langle
f,\mathcal{L}_n\rangle\}_n} \sss$, where the first mapping is the
canonical inclusion. Hence, $\iota$ has a closed graph in
$G^{\alpha}_{\alpha}(\RR^d_+)\times \sss$. Since the range of
$\iota$ is in $\sss^{\alpha}$ and $\sss^{\alpha}$ is continuously
injected into $\sss$, the graph of $\iota$ is closed in
$G^{\alpha}_{\alpha}(\RR^d_+)\times \sss^{\alpha}$.
$G^{\alpha}_{\alpha}(\RR^d_+)$ is injective inductive limit of
$(F)$-spaces. For this reason, $G^{\alpha}_{\alpha}(\RR^d_+)$ is
ultrabornological (cf. \cite[Theorem 7, p. 72]{kothe2}; every
$(F)$-space is ultrabornological). Moreover, $\sss^{\alpha}$ is a
webbed space of De Wilde (see \cite[Theorem 11, p. 64]{kothe2}).
Hence, the closed graph theorem of De Wilde (see \cite[Theorem 2,
p. 57]{kothe2}) implies that $\iota:
G^{\alpha}_{\alpha}(\RR^d_+)\rightarrow \sss^{\alpha}$ is
continuous. Also, $\sss^{\alpha}$ is ultrabornological since it is
bornological and complete and $G^{\alpha}_{\alpha}(\RR^d_+)$ is a
webbed space of De Wilde (cf. \cite[Theorem 8, p. 63]{kothe2};
every $(F)$-space is a webbed space of De Wilde). The mapping
$\iota^{-1}:\sss^{\alpha}\rightarrow
G^{\alpha}_{\alpha}(\RR^d_+)$, which has a closed graph, is
continuous by the De Wilde closed graph theorem (see \cite[Theorem
2, p. 57]{kothe2}). Now, Proposition \ref{lkrrr} implies that
$G^{\alpha}_{\alpha}(\RR^d_+)$ is a $(DFN)$-space and
$(G^{\alpha}_{\alpha}(\RR^d_+))'$ is an $(FN)$-space.\\
\indent Given $f\in G^{\alpha}_{\alpha}(\RR^d_+)$, let
$a_n=\langle f,\mathcal{L}_n\rangle$. For each finite
$\Phi\subseteq \NN^d_0$, denote $f_{\Phi}= \sum_{n\in \Phi}
a_n\mathcal{L}_n\in G^{\alpha}_{\alpha}(\RR^d_+)$ (since
$\mathcal{L}_n\in G^{\alpha}_{\alpha}(\RR^d_+)$). Let $a>1$ be
such that $\iota(f)\in \sss^{\alpha,a}$. Fix $1<a'<a$. One easily
verifies that for each $\varepsilon>0$ there exists finite
$\Phi_0\subseteq \NN^d_0$ such that for each finite
$\Phi\subseteq\NN^d_0$, satisfying $\Phi_0\subseteq \Phi$, we have
$\|\iota(f)-\iota(f_{\Phi})\|_{\sss^{\alpha,a'}}\leq \varepsilon$.
Since $\iota$ is an isomorphism this implies that for each
neighbourhood of zero $V\subseteq G^{\alpha}_{\alpha}(\RR^d_+)$
there exists finite $\Phi_0\subseteq \NN^d_0$ such that for finite
$\Phi\supseteq \Phi_0$ we have $f-f_{\Phi}\in V$, i.e.
$\sum_{n\in\NN^d_0} a_n\mathcal{L}_n$ is summable to $f$ in
$G^{\alpha}_{\alpha}(\RR^d_+)$. \qed
\end{proof}

\begin{thm}\label{D ocena koef za dual}
Let $\alpha\geq 1$. The mapping $\tilde{\iota}:
(G^{\alpha}_{\alpha}(\RR^d_+))'\rightarrow (\sss^{\alpha})'$,
$\tilde{\iota}(T)=\{\langle
T,\mathcal{L}_n\rangle\}_{n\in\NN^d_0}$, is a topological
isomorphism. Moreover, $\sum_{n\in\NN^d_0}\langle
T,\mathcal{L}_n\rangle \mathcal{L}_n$ is summable to $T$ in
$(G^{\alpha}_{\alpha}(\RR^d_+))'$.
\end{thm}

\begin{proof} By Theorem \ref{isoss}, both the transpose of $\iota$, ${}^t\iota:(\sss^{\alpha})'\rightarrow
(G^{\alpha}_{\alpha}(\RR^d_+))'$, and its inverse
$({}^t\iota)^{-1}:(G^{\alpha}_{\alpha}(\RR^d_+))'\rightarrow
(\sss^{\alpha})'$ are topological isomorphisms. For $T\in
(G^{\alpha}_{\alpha}(\RR^d_+))'$, let
$\{b_n\}_n=({}^t\iota)^{-1}(T)$. Then \beas \langle
T,\mathcal{L}_n\rangle=\langle
{}^t\iota(\{b_n\}_n),\mathcal{L}_n\rangle=\langle
\{b_n\}_n,\iota(\mathcal{L}_n)\rangle= b_n. \eeas Thus $\{\langle
T,\mathcal{L}_n\rangle\}_n=\{b_n\}_n=({}^t\iota)^{-1}(T)\in(\sss^{\alpha})'$.
Hence $\tilde{\iota}$ is in fact a topological isomorphism
$({}^t\iota)^{-1}: (G^{\alpha}_{\alpha}(\RR^d_+))'\rightarrow
(\sss^{\alpha})'$. Now, by the similar approach as above, one
proves that $\sum_{n\in\NN^d_0}\langle T,\mathcal{L}_n\rangle
\mathcal{L}_n$ is summable to $T$ in
$(G^{\alpha}_{\alpha}(\RR^d_+))'$. \qed
\end{proof}

For $T\in (G^{\alpha}_{\alpha}(\RR^d_+))'$, by Lemma
\ref{analyticitt},
$\mathcal{F}_{\mathbf{D}}T\in\mathcal{O}(\mathbf{D})$. Since
$\sum_n \langle T,\mathcal{L}_n\rangle \mathcal{L}_n$ is summable
to $T$ in $(G^{\alpha}_{\alpha}(\RR^d_+))'$, by the same method as
in the proof of Proposition \ref{Theorem C}, one proves the
following result.

\begin{prop}
Let $T\in(G^{\alpha}_{\alpha}(\mathbb{R}_+^d))'$, $\alpha\geq 1$
and $b_n=\langle T,\mathcal{L}_n\rangle$, $n\in\mathbb{N}^d_0$.
Then, \beas
\mathcal{F}_\mathbf{D}(T)(w)=\prod_{j=1}^d(1-w_j)\sum_{n\in\mathbb{N}^d_0}b_{n}w^{n},\;w\in
\mathbf{D}. \eeas In particular, if $\mathcal{F}_{\mathbf{D}}T=0$
then $T=0$.
\end{prop}

Now, we state the kernel theorems:

\begin{thm}\label{Thm. tenzorski proizvod G }
Let $\alpha\geq 1$. We have the following canonical isomorphism:
$$G_\alpha^\alpha(\mathbb{R}^{d_1}_+)\hat{\otimes}
G_\alpha^\alpha(\mathbb{R}^{d_2}_+)\cong
G_\alpha^\alpha(\mathbb{R}^{d_1+d_2}_+),\,\,
(G_\alpha^\alpha(\mathbb{R}^{d_1}_+))'\hat{\otimes}
(G_\alpha^\alpha(\mathbb{R}^{d_2}_+))'\cong
(G_\alpha^\alpha(\mathbb{R}^{d_1+d_2}_+))'.$$
\end{thm}

\begin{proof}
For simplicity, put $d=d_1+d_2$. Let $\sss^{\alpha}_{d_1}$,
$\sss^{\alpha}_{d_2}$ and $\sss^{\alpha}$ be the
$d_1$-dimensional, the $d_2$-dimensional and the $d$-dimensional
variant of the space $\sss^{\alpha}$, respectively. Firstly, we
prove that
$(\sss^{\alpha}_{d_1})'\hat{\otimes}(\sss^{\alpha}_{d_2})'\cong
(\sss^{\alpha})'$, where an isomorphism is given by the extension
of the canonical inclusion
$(\sss^{\alpha}_{d_1})'\otimes(\sss^{\alpha}_{d_2})'\rightarrow
(\sss^{\alpha})'$, $\{u_n\}_{n\in\NN^{d_1}_0}\otimes
\{v_m\}_{m\in\NN^{d_2}_0}\mapsto \{u_nv_m\}_{(n,m)\in\NN^d_0}$.
Observe that the mapping

$$\left(\{u_n\}_{n\in\NN^{d_1}_0},
\{v_m\}_{m\in\NN^{d_2}_0}\right)\mapsto
\{u_nv_m\}_{(n,m)\in\NN^d_0}, (\sss^{\alpha}_{d_1})'\times
(\sss^{\alpha}_{d_2})'\rightarrow (\sss^{\alpha})'$$

\noindent is continuous. Hence, the $\pi$ topology on
$(\sss^{\alpha}_{d_1})'\otimes (\sss^{\alpha}_{d_2})'$ is stronger
than the induced one from $(\sss^{\alpha})'$. Let $A$ and $B$ be
the equicontinuous subsets of
$((\sss^{\alpha}_{d_1})')'=\sss^{\alpha}_{d_1}$ and
$((\sss^{\alpha}_{d_2})')'=\sss^{\alpha}_{d_2}$, respectively
($\sss^{\alpha}$ is reflexive since it is a $(DFN)$-space). Hence,
there exist $C>0$ and $r>1$ such that \beas |\langle
\{u_n\}_{n\in\NN^{d_1}_0},\{a_n\}_{n\in\NN^{d_1}_0}\rangle| &\leq&
C\sum_{n\in\NN^{d_1}_0}|u_n|r^{-|n|^{1/\alpha}},\,\,
\forall \{a_n\}_{n\in\NN^{d_1}_0}\in A,\, \forall\{u_n\}_{n\in\NN^{d_1}_0}\in (\sss^{\alpha}_{d_1})'\,\\
|\langle
\{v_m\}_{m\in\NN^{d_2}_0},\{b_m\}_{m\in\NN^{d_2}_0}\rangle| &\leq&
C\sum_{m\in\NN^{d_2}_0}|v_m|r^{-|m|^{1/\alpha}},\,\,
\{b_m\}_{m\in\NN^{d_2}_0}\in B,\, \{v_m\}_{m\in\NN^{d_2}_0}\in
(\sss^{\alpha}_{d_2})'. \eeas Let
$\{\chi_{(n,m)}\}_{(n,m)\in\NN^d_0}=\sum_{j=1}^l\{u^{(j)}_n\}_{n\in\NN^{d_1}_0}\otimes\{v^{(j)}_m\}_{m\in\NN^{d_2}_0}\in
(\sss^{\alpha}_{d_1})'\otimes(\sss^{\alpha}_{d_2})'$. Then, for
$\{a_n\}_{n\in\NN^{d_1}_0}\in A$ and
$\{b_m\}_{m\in\NN^{d_2}_0}\in B$, we have\\
\\
$\left|\left\langle
\{\chi_{(n,m)}\}_{(n,m)},\{a_n\}_n\otimes\{b_m\}_m\right\rangle\right|$
\beas &=&\left|\left\langle
\sum_{j=1}^l\{u^{(j)}_n\}_n\left\langle
\{v^{(j)}_m\}_m,\{b_m\}_m\right\rangle,
\{a_n\}_n\right\rangle\right|=\left|\left\langle
\left\{\sum_{j=1}^l\left\langle
\{v^{(j)}_m\}_m,\{b_m\}_m\right\rangle u^{(j)}_n\right\}_n,
\{a_n\}_n\right\rangle\right|\\
&\leq& C\sum_{n\in\NN^{d_1}_0}\left|\sum_{j=1}^l \left\langle
\{v^{(j)}_m\}_m,\{b_m\}_m\right\rangle u^{(j)}_n\right|
r^{-|n|^{1/\alpha}}\\
&=&C\sum_{n\in\NN^{d_1}_0}\left|\left\langle \left\{\sum_{j=1}^l
u^{(j)}_nv^{(j)}_m\right\}_m,\{b_m\}_m\right\rangle \right|
r^{-|n|^{1/\alpha}} \leq C^2\sum_{(n,m)\in\NN^d}\left|\sum_{j=1}^l u^{(j)}_nv^{(j)}_m\right|r^{-|n|^{1/\alpha}-|m|^{1/\alpha}}\\
&\leq&
C^2\left\|\{\chi_{(n,m)}\}_{(n,m)}\right\|_{(\sss^{\alpha})',r}.
\eeas We can conclude that the $\epsilon$ topology on
$(\sss^{\alpha}_{d_1})'\otimes (\sss^{\alpha}_{d_2})'$ is weaker
than the induced one from $(\sss^{\alpha})'$. Since
$(\sss^{\alpha})'$ is nuclear, these topologies are identical.
Clearly, $(\sss^{\alpha}_{d^1})'\otimes (\sss^{\alpha}_{d_2})'$ is
dense in $(\sss^{\alpha}_d)'$. Hence, we proved the desired
topological isomorphism. As all spaces in consideration are
$(FN)$-spaces, by duality we have
$\sss^{\alpha}_{d_1}\hat{\otimes} \sss^{\alpha}_{d_2}\cong
\sss^{\alpha}$. Note that the isomorphism is in fact the extension
of the canonical inclusion $\kappa:\sss^{\alpha}_{d_1}\otimes
\sss^{\alpha}_{d_2}\rightarrow \sss^{\alpha}$,
$\kappa(\{a_n\}_n\otimes \{b_m\}_m)=\{a_nb_m\}_{(n,m)}$. Now
observe that the diagram
\begin{center}
\begin{tikzpicture}[normal line/.style={->}]
\matrix (m) [matrix of math nodes, row sep=3em,
column sep=2.5em, text height=1.5ex, text depth=0.25ex]
{\sss^{\alpha}_{d_1}\otimes\sss^{\alpha}_{d_2} & \sss^{\alpha} \\
G^{\alpha}_{\alpha}(\RR^{d_1}_+)\otimes G^{\alpha}_{\alpha}(\RR^{d_2}_+) & G^{\alpha}_{\alpha}(\RR^d_+) \\ };
\path[normal line]
(m-1-1) edge  node [auto] {$\kappa$} (m-1-2)
(m-2-1) edge  node [auto] {$\iota\otimes \iota$}(m-1-1)

(m-2-1) edge (m-2-2)
(m-2-2) edge node [auto] {$\iota$}(m-1-2);
\end{tikzpicture}
\end{center}
commutes, where the bottom horizontal
line is the canonical inclusion $f\otimes g(x,y)\mapsto f(x)g(y)$.
Since $\kappa$ extends to an isomorphism, by Theorem \ref{isoss},
it follows that the canonical inclusion
$G^{\alpha}_{\alpha}(\RR^{d_1}_+)\otimes
G^{\alpha}_{\alpha}(\RR^{d_2}_+)\rightarrow
G^{\alpha}_{\alpha}(\RR^d_+)$ is continuous and it extends to an
isomorphism $G^{\alpha}_{\alpha}(\RR^{d_1}_+)\hat{\otimes}
G^{\alpha}_{\alpha}(\RR^{d_2}_+)\cong
G^{\alpha}_{\alpha}(\RR^d_+)$. The assertion
$(G^{\alpha}_{\alpha}(\RR^{d_1}_+))'\hat{\otimes}
(G^{\alpha}_{\alpha}(\RR^{d_2}_+))'\cong
(G^{\alpha}_{\alpha}(\RR^d_+))'$ can be obtained by the duality of
an isomorphism $G^{\alpha}_{\alpha}(\RR^{d_1}_+)\hat{\otimes}
G^{\alpha}_{\alpha}(\RR^{d_2}_+)\cong
G^{\alpha}_{\alpha}(\RR^d_+)$ since $G^{\alpha}_{\alpha}(\RR^d_+)$
is a $(DFN)$-space. \qed
\end{proof}

\section{Weyl pseudo-differential operators with radial symbols from the $G$-type spaces and their duals}

Concerning pseudo-differential operators, especially the Weyl
calculus, we refer to the standard books \cite{NR} and
\cite{Shubin}, for example.

Let $f,g\in \SSS(\mathbb{R}^d)$. Then the function $W(f,g)$
defined on $\mathbb{R}^{2d}$ by

$$W(f,g)(x,\xi)=(2\pi)^{-d/2}\int_{\mathbb{R}^d}e^{-i\xi\cdot p}f\left(x+\frac{p}{2}\right)
\overline{g\left(x-\frac{p}{2}\right)}dp,\quad
x,\xi\in\mathbb{R}^d$$

\noindent is called the Wigner transform of $f$ and $g$. The
bilinear mapping $(f,g)\mapsto W(f,\overline{g})$,
$\SSS(\mathbb{R}^d)\times \SSS(\mathbb{R}^d)\rightarrow
\SSS(\mathbb{R}^{2d})$ is continuous and (cf. \cite[Corollary
3.4]{Wong}) can be extended uniquely to a bilinear operator
$L^2(\mathbb{R}^d)\times L^2(\mathbb{R}^d)\rightarrow
L^2(\mathbb{R}^{2d})$ such that

$$\|W(f,g)\|_{L^2(\mathbb{R}^{2d})}=\|f\|_{L^2(\mathbb{R}^{d})}\|g\|_{L^2(\mathbb{R}^{d})},\;f,g\in L^2(\mathbb{R}^{d}).$$

\noindent Let $f,g\in \SSS_\alpha^\alpha(\mathbb{R}^d)$,
$\alpha\geq 1/2$. By \cite[Theorem 3.8, p. 179]{teofanov},
$W(f,g)\in \SSS_\alpha^\alpha(\mathbb{R}^{2d})$. Moreover, we have
the following proposition.

\begin{prop}\label{concccwig}
A bilinear mapping $(f,g)\mapsto W(f,\overline{g})$,
$\SSS^{\alpha}_{\alpha}(\RR^d)\times \SSS^{\alpha}_{\alpha}(\RR^d)
\rightarrow\SSS^{\alpha}_{\alpha}(\RR^{2d})$, is continuous.
\end{prop}

\begin{proof} Fix $g\in\SSS^{\alpha}_{\alpha}(\RR^d)$. If we consider a mapping $f\mapsto W(f,\overline{g})$ as a mapping
from $\SSS^{\alpha}_{\alpha}(\RR^d)$ into $\SSS(\RR^{2d})$ it is
continuous since it decomposes as $\ds
\SSS^{\alpha}_{\alpha}(\RR^d)\longrightarrow
\SSS(\RR^d)\xrightarrow{f\mapsto W(f,\overline{g})}
\SSS(\RR^{2d})$, where the first mapping is the canonical
inclusion. Hence, its graph is closed in
$\SSS^{\alpha}_{\alpha}(\RR^d)\times \SSS(\RR^{2d})$. Since its
image is in $\SSS^{\alpha}_{\alpha}(\RR^{2d})$, its graph is
closed in $\SSS^{\alpha}_{\alpha}(\RR^d)\times
\SSS^{\alpha}_{\alpha}(\RR^{2d})$. As
$\SSS^{\alpha}_{\alpha}(\RR^d)$ is a $(DFS)$-space it is an
ultrabornological and webbed space of De Wilde (cf. \cite[Theorem
11, p. 64]{kothe2}). Now, the De Wilde closed graph theorem (see
\cite[Theorem 2, p. 57]{kothe2}) implies its continuity.
Similarly, for each fixed $f\in\SSS^{\alpha}_{\alpha}(\RR^d)$, the
mapping $g\mapsto W(f,\overline{g})$,
$\SSS^{\alpha}_{\alpha}(\RR^d)\rightarrow
\SSS^{\alpha}_{\alpha}(\RR^{2d})$ is continuous. Thus the bilinear
mapping $(f,g)\mapsto W(f,\overline{g})$,
$\SSS^{\alpha}_{\alpha}(\RR^d)\times \SSS^{\alpha}_{\alpha}(\RR^d)
\rightarrow\SSS^{\alpha}_{\alpha}(\RR^{2d})$, is separately
continuous and hence continuous since
$\SSS^{\alpha}_{\alpha}(\RR^d)$ is barrelled $(DF)$-space (cf.
\cite[Theorem 11, p. 161]{kothe2}). \qed
\end{proof}

Recall, for $\sigma\in \SSS(\mathbb{R}^{2d})$ the Weyl
pseudo-differential operator with symbol $\sigma$ is defined by
\begin{equation}\label{Wong Weyl pseudo-differential operators}
(W_\sigma
f)(x)=(2\pi)^{-d}\int_{\mathbb{R}^d}\int_{\mathbb{R}^d}e^{i(x-y)\cdot\xi}\sigma\left(\frac{x+y}{2},\xi\right)f(y)dyd\xi,\quad
f\in \SSS(\mathbb{R}^d).
\end{equation}
\noindent and can be extended on $(\SSS(\mathbb{R}^d))'$ as a
linear and continuous operator from $(\SSS(\mathbb{R}^d))'$ into
itself. Let $\alpha\geq 1/2$. The Weyl pseudo-differential
operator with a symbol
$\sigma\in(\SSS_\alpha^\alpha(\mathbb{R}^{2d}))'$ defined by
\begin{equation}\label{Wong Th.12.1}
(W_\sigma
f)(g)=(2\pi)^{-d/2}\langle\sigma,W(f,\overline{g})\rangle
\end{equation}
is a continuous and linear mapping from
$\SSS_\alpha^\alpha(\mathbb{R}^d)$ into
$(\SSS_\alpha^\alpha(\mathbb{R}^d))'$.\\
\indent Our goal is to analyse the Weyl pseudo-differential
operator on $\SSS^{\alpha}_{\alpha}(\RR^d)$
and $(\SSS^{\alpha}_{\alpha}(\RR^d))'$, $\alpha\geq 1/2$, when its symbol originates from $(G^{2\alpha}_{2\alpha}(\RR^d_+))'$, $\alpha\geq 1/2$.\\
\indent Throughout the rest of this section, we denote by $v$ the
mapping $\RR^{2d}\rightarrow \overline{\RR^d_+}$, $(x,\xi)\mapsto
v(x,\xi)=(x_1^2+\xi_1^2,\ldots,x_d^2+\xi_d^2)$.

\begin{prop}\label{Faa di Bruno Prop}
Let $\sigma\in\mathcal{S}(\mathbb{R}^d_+)$. Then
$\tilde{\sigma}(x,\xi)=\sigma\circ v(x,\xi)\in
\mathcal{S}(\mathbb{R}^{2d})$. Moreover, the mapping
$\sigma\mapsto\tilde{\sigma}=\sigma\circ v$,
$\SSS(\RR^d_+)\rightarrow \SSS(\RR^{2d})$, is continuous.
\end{prop}

\begin{proof}
Fix $j\in\NN$. For $p,q\in\NN^d_0$, $|p|\leq j$ and $|q|\leq j$
observe that $D^p_x D^q_{\xi}\tilde{\sigma}(x,\xi)$ is a finite
sum of the form $P(x,\xi)D^{p'}_x D^{q'}_{\xi}\sigma(v(x,\xi))$,
where $P(x,\xi)$ are polynomials in $(x,\xi)$ of degree at most
$|p|+|q|$ which do not depend on $\sigma$ (they only depend on the
derivatives of $v$) and $p',q'\in\NN^d_0$ are such that $p'\leq p$
and $q'\leq q$. Moreover, observe that the number of such terms
that appear in $D^p_x D^q_{\xi}\tilde{\sigma}(x,\xi)$ depend only
on $p$ and $q$ (and not on $\sigma$). For $p'',q''\in\NN^d_0$ we
also have $\left|x^{p''}\xi^{q''}\right|\leq
|x|^{|p''|}|\xi|^{|q''|}\leq (|x|^2+|\xi|^2)^{(|p''|+|q''|)/2}$.
Thus,
$$\sup_{\substack{|p''|\leq j\\ |q''|\leq j}}\sup_{\substack{|p|\leq j\\ |q|\leq j}}
\sup_{(x,\xi)\in\RR^{2d}}\left|x^{p''}\xi^{q''}D^p_x D^q_{\xi}\tilde{\sigma}(x,\xi)\right|\leq C\sup_{\substack{|n|\leq 2j\\
|m|\leq 2j}}\sup_{t\in\RR^d_+}\left|t^m D^n \sigma(t)\right|.$$
Hence, $\tilde{\sigma}\in\mathcal{S}(\mathbb{R}^{2d})$ and the
mapping $\sigma\mapsto\tilde{\sigma}=\sigma\circ v$,
$\SSS(\RR^d_+)\rightarrow \SSS(\RR^{2d})$ is continuous. \qed
\end{proof}

\indent Let $\alpha\geq 1/2$ and $\sigma(\rho)\in
G^{2\alpha}_{2\alpha}(\mathbb{R}^d_+)$. Denote by
$\sigma_0(\rho)=\sigma(2\rho)$, $\rho\in\RR^d_+$. Then the
functions $\tilde{\sigma}$ and $\tilde{\sigma}_0$ defined by
\bea\label{symff} \tilde{\sigma}(x,\xi)=\sigma\circ v(x,\xi),\,\,
\tilde{\sigma}_0(x,\xi)=\sigma_0\circ
v(x,\xi),\;(x,\xi)\in\mathbb{R}^{2d} \eea belong to
$\SSS(\RR^{2d})$ (see Proposition \ref{Faa di Bruno Prop}). Hence,
the Weyl pseudo-differential operator with a symbol
$\tilde{\sigma}_0$ is a continuous mapping
$\SSS^{\alpha}_{\alpha}(\RR^d)\rightarrow
(\SSS^{\alpha}_{\alpha}(\RR^d))'$.

\begin{thm}\label{ter11}
Let $\alpha\geq 1/2$ and $\sigma(\rho)\in
G^{2\alpha}_{2\alpha}(\mathbb{R}^d_+)$. Denote by
$\sigma_0(\rho)=\sigma(2\rho)$, $\rho\in\RR^d_+$. Let
$\tilde{\sigma},\tilde{\sigma}_0\in\SSS(\RR^{2d})$ be the
functions defined in (\ref{symff}). Then the Weyl
pseudo-differential operator $W_{\tilde{\sigma}_0}$ is a
continuous operator $\SSS^{\alpha}_{\alpha}(\RR^d)\rightarrow
\SSS^{\alpha}_{\alpha}(\RR^d)$ and it extends to a continuous
mapping $W_{\tilde{\sigma}_0}:(\SSS^{\alpha}_{\alpha}(\RR^d))'
\rightarrow \SSS^{\alpha}_{\alpha}(\RR^d)$. If
$f,g\in(\mathcal{S}^{\alpha}_{\alpha}(\mathbb{R}^d))'$ and
$$f_k=\langle f,h_k\rangle, g_k=\langle g,h_k\rangle\;\mbox{and}\;
\sigma_k=(2\pi)^{d/2}(-1)^{|k|}2^{-d}
\int_{\mathbb{R}^d_+}\sigma(\rho)\mathcal{L}_{k}(\rho)d\rho,$$
then
$(W_{\tilde{\sigma}_0}f)(g)=(2\pi)^{-d/2}\sum_{k\in\mathbb{N}^d_0}f_kg_k\sigma_k$.
Moreover, if
$\sigma_{0,j}(\eta)\rightarrow\sigma_{0}(\eta)\;\mbox{in}\;
G^{2\alpha}_{2\alpha}(\mathbb{R}^d_+)$ as $j\rightarrow\infty$
then $W_{\tilde{\sigma}_{0,j}}\rightarrow W_{\tilde{\sigma}_0}$ in
the strong topology of
$\mathcal{L}((\SSS^{\alpha}_{\alpha}(\RR^d))',\SSS^{\alpha}_{\alpha}(\RR^d))$.
\end{thm}

\begin{proof}
First we compute the Weyl pseudo-differential transform
$W_{\tilde{\sigma}_0}$ of $f,g\in\SSS^{\alpha}_{\alpha}(\RR^d)$.
Since $\sum_{n\in\NN^d_0} f_nh_n$ and $\sum_{n\in\NN^d_0} g_nh_n$
converge absolutely to $f$ and $g$ in
$\SSS^{\alpha}_{\alpha}(\RR^d)$ respectively (cf. Proposition
\ref{herexpult}) and the mapping $(\varphi,\psi)\mapsto
W(\varphi,\overline{\psi})$, $\SSS^{\alpha}_{\alpha}(\RR^d)\times
\SSS^{\alpha}_{\alpha}(\RR^d)\rightarrow
\SSS^{\alpha}_{\alpha}(\RR^{2d})$, is continuous (see Proposition
\ref{concccwig}), we conclude
$W(f,\overline{g})=\sum_{(m,k)\in\NN^{2d}_0}f_m g_kW(h_m,h_k)$,
where the sum converges absolutely in
$\SSS^{\alpha}_{\alpha}(\RR^{2d})$. As
$\tilde{\sigma}_0\in\SSS(\RR^{2d})\subseteq
(\SSS^{\alpha}_{\alpha}(\RR^{2d}))'$, we have
\begin{eqnarray}
(W_{\tilde{\sigma}_0}
f)(g)=(2\pi)^{-d/2}\sum_{(m,k)\in\mathbb{N}^{2d}_0}f_mg_k
\langle\tilde{\sigma}_0,\psi_{m,k}\rangle, \label{Wong 24.6 1
dimm}
\end{eqnarray}
where $\psi_{m,k}=W(h_m,h_k)$. Clearly,
$\psi_{m,k}=\prod_{r=1}^d\psi_{m_r,k_r}$, where
$\psi_{m_r,k_r}=W(h_{m_r},h_{k_r})$. Using \cite[Theorem
24.1]{Wong} and denoting $\eta_r=x_r+i\xi_r\in \mathbb{C}$, we
have

$$\psi_{m_r,k_r}(x_r,\xi_r)=2(-1)^{k_r}(2\pi)^{-1/2}
\left(\frac{k_r!}{m_r!}\right)^{1/2}
(\sqrt{2})^{m_r-k_r}(\overline{\eta_r})^{m_r-k_r}L_{k_r}^{m_r-k_r}(2|\eta_r|^2)e^{-|\eta_r|^2},
m_r\geq k_r,$$
$$\psi_{m_r,k_r}(x_r,\xi_r)=2(-1)^{m_r}(2\pi)^{-1/2}
\left(\frac{m_r!}{k_r!}\right)^{1/2}
(\sqrt{2})^{k_r-m_r}\eta_r^{k_r-m_r}L_{m_r}^{k_r-m_r}(2|\eta_r|^2)e^{-|\eta_r|^2},\,
\mbox{if}\, k_r\geq m_r.$$

\noindent In terms of polar coordinates the integral
$\langle\tilde{\sigma}_0,\psi_{m,k}\rangle =
\int_{\mathbb{R}^{2d}}\sigma_0(v(x,\xi))\psi_{m,k}(x,\xi)dxd\xi$
is
$$\langle\tilde{\sigma}_0,\psi_{m,k}\rangle= C_{m,k}\prod_{r=1}^d\int_{-\pi}^{\pi}e^{-i(m_r-k_r)\theta_r} d\theta_r.$$
Thus $\langle\tilde{\sigma}_0,\psi_{m,k}\rangle=0$ when $m\neq k$.
Moreover,
\begin{eqnarray*}
\langle\tilde{\sigma}_0,\psi_{k,k}\rangle&=&
(2\pi)^{d/2}(-1)^{|k|}2^d\int_{\mathbb{R}_+^{d}}\sigma(2\rho^2_1,\ldots,2\rho^2_d)
L_{k}(2\rho^2_1,\ldots,2\rho^2_d)e^{-|\rho|^2}\rho^{\mathbf{1}}
d\rho\\
&=&(2\pi)^{d/2}(-1)^{|k|}2^{-d}
\int_{\mathbb{R}_+^{d}}\sigma(y)\mathcal{L}_{k}(y)dy=\sigma_k.
\end{eqnarray*}
By (\ref{Wong 24.6 1 dimm}), we obtain
\begin{equation}\label{trttt}
(W_{\tilde{\sigma}_0}
f)(g)=(2\pi)^{-d/2}\sum_{k\in\mathbb{N}_0^d}f_kg_k\sigma_k
\end{equation}
and the series converges absolutely since
$\{f_n\}_n,\{g_n\}_n,\{\sigma_n\}_n\in \sss^{2\alpha}$
($f,g\in\SSS^{\alpha}_{\alpha}(\RR^d)$, $\sigma\in
G^{2\alpha}_{2\alpha}(\RR^d_+)$). Let now
$f,g\in(\SSS^{\alpha}_{\alpha}(\RR^d))'$. Define
$(W_{\tilde{\sigma}_0}f)(g)=(2\pi)^{-d/2}\sum_{n\in
\NN^d_0}f_ng_n\sigma_n$. Observe that the series converges
absolutely since $\{f_n\}_{n\in\NN^d_0},\{g_n\}_{n\in\NN^d_0}\in
(\sss^{2\alpha})'$ and $\{\sigma_n\}_{n\in\NN^d_0}\in
\sss^{2\alpha}$ ($\sigma\in
G^{2\alpha}_{2\alpha}(\mathbb{R}^d_+)$; cf. Theorem \ref{isoss}).
Thus, if we fix $f\in(\SSS^{\alpha}_{\alpha}(\RR^d))'$, the
mapping $g\mapsto (W_{\tilde{\sigma}_0} f)(g)$,
$(\SSS^{\alpha}_{\alpha}(\RR^d))'\rightarrow \CC$, is a well
defined linear mapping. To prove that it is continuous let $B$ be
a bounded subset of $(\SSS^{\alpha}_{\alpha}(\RR^d))'$. Thus for
each $a>1$ there exists $C>0$ such that $|g_k|\leq C
a^{|k|^{1/(2\alpha)}}$, $\forall k\in\NN^d_0$, $\forall g\in B$.
Hence,
$$\sup_{g\in B}\left|(W_{\tilde{\sigma}_0}f)(g)\right|<\infty,$$
i.e. $W_{\tilde{\sigma}_0}f$ maps bounded subsets in
$(\SSS^{\alpha}_{\alpha}(\RR^d))'$ into bounded subsets of $\CC$.
Since $(\SSS^{\alpha}_{\alpha}(\RR^d))'$ is bornological,
$g\mapsto (W_{\tilde{\sigma}_0}f)(g)$ is continuous, hence
$W_{\tilde{\sigma}_0}f\in\SSS^{\alpha}_{\alpha}(\RR^d)$
($\SSS^{\alpha}_{\alpha}(\RR^d)$ is reflexive). Now we conclude
that $W_{\tilde{\sigma}_0}f=\sum_{n\in\NN^d_0}f_n\sigma_n h_n$
(this is exactly Hermite expansion of $W_{\tilde{\sigma}_0}f$;
$\{f_n\sigma_n\}_n\in \sss^{2\alpha}$). Thus, the mapping
$f\mapsto W_{\tilde{\sigma}_0}f$,
$(\SSS^{\alpha}_{\alpha}(\RR^d))'\rightarrow
\SSS^{\alpha}_{\alpha}(\RR^d)$, is well defined and linear.
Arguing similarly as before, one can prove that when $f$ varies in
a bounded subset $B$ of $(\SSS^{\alpha}_{\alpha}(\RR^d))'$, the
set $\{\{f_k\sigma_k\}_{k\in\NN^d_0}|\, f\in B\}$ is bounded in
$\sss^{2\alpha}$. Thus $\{W_{\tilde{\sigma}_0}f|\, f\in B\}$ is
bounded in $\SSS^{\alpha}_{\alpha}(\RR^d)$. As
$(\SSS^{\alpha}_{\alpha}(\RR^d))'$ is bornological, the mapping
$f\mapsto W_{\tilde{\sigma}_0} f$,
$(\SSS^{\alpha}_{\alpha}(\RR^d))'\rightarrow
\SSS^{\alpha}_{\alpha}(\RR^d)$, is continuous. Observe that
$W_{\tilde{\sigma}_0} f$ coincides with the Weyl transform of $f$
when $f\in\SSS^{\alpha}_{\alpha}(\RR^d)$ (cf. (\ref{trttt})).\\
\indent If $\sigma_j\rightarrow \sigma$ as $j\rightarrow \infty$,
in $G^{2\alpha}_{2\alpha}(\RR^d_+)$, Theorem \ref{isoss} implies
that $\{\sigma_{n,j}\}_{n\in \NN^d_0}\rightarrow
\{\sigma_n\}_{n\in\NN^d}$ as $j\rightarrow \infty$ in
$\sss^{2\alpha}$ and since the latter is a $(DFN)$-space, the
convergence also holds in $\sss^{2\alpha,a}$ for some $a>1$. Thus,
for each fixed $f\in(\SSS^{\alpha}_{\alpha}(\RR^d))'$,
$\{f_n\sigma_{n,j}\}_n\rightarrow \{f_n\sigma_n\}_n$ in
$\sss^{2\alpha}$. Hence,
$\sum_{n\in\NN^d_0}f_n\sigma_{n,j}h_n\rightarrow
\sum_{n\in\NN^d_0}f_n\sigma_nh_n$ in
$\SSS^{\alpha}_{\alpha}(\RR^d)$. We obtain
$W_{\tilde{\sigma}_{0,j}}\rightarrow W_{\tilde{\sigma}_0}$ in the
topology of simple convergence in
$\mathcal{L}((\SSS^{\alpha}_{\alpha}(\RR^d))',\SSS^{\alpha}_{\alpha}(\RR^d))$.
Now, the Banach-Steinhaus theorem implies that the convergence
holds in the topology of precompact convergence. Since
$(\SSS^{\alpha}_{\alpha}(\RR^d))'$ is a Montel space, the
convergence also holds in the strong topology of
$\mathcal{L}((\SSS^{\alpha}_{\alpha}(\RR^d))',\SSS^{\alpha}_{\alpha}(\RR^d))$.
\qed
\end{proof}

By similar arguments, one can prove the following theorem.

\begin{thm}\label{ter17}
Let $\sigma(\rho)\in\mathcal{S}(\mathbb{R}^d_+)$ and denote by
$\sigma_0(\rho)=\sigma(2\rho)$, $\rho\in\RR^d_+$. Let
$\tilde{\sigma},\tilde{\sigma}_0\in\SSS(\RR^{2d})$ be the
functions defined in (\ref{symff}). Then the Weyl
pseudo-differential operator $W_{\tilde{\sigma}_0}$ extends to a
continuous mapping $W_{\tilde{\sigma}_0}:(\SSS(\RR^d))'\rightarrow
\SSS(\RR^d)$. If $f,g\in(\mathcal{S}(\mathbb{R}^d))'$ and
$$f_k=\langle f,h_k\rangle, g_k=\langle g,h_k\rangle\;\mbox{and}\;
\sigma_k=(2\pi)^{d/2}(-1)^{|k|}2^{-d}
\int_{\mathbb{R}^d_+}\sigma(\rho)\mathcal{L}_{k}(\rho)d\rho,$$
then
$(W_{\tilde{\sigma}_0}f)(g)=(2\pi)^{-d/2}\sum_{k\in\mathbb{N}^d_0}f_kg_k\sigma_k$.
Moreover, if
$\sigma_{0,j}(\eta)\rightarrow\sigma_{0}(\eta)\;\mbox{in}\;
\mathcal{S}(\mathbb{R}^d_+)$ as $j\rightarrow\infty$ then
$W_{\tilde{\sigma}_{0,j}}\rightarrow W_{\tilde{\sigma}_0}$ in the
strong topology of $\mathcal{L}((\SSS(\RR^d))',\SSS(\RR^d))$.
\end{thm}

Let $\alpha\geq 1/2$. If $\sigma$ is a measurable function on
$\RR^d_+$ such that $\sigma(\rho)/(\mathbf{1}+\rho)^{n/2}\in
L^2(\RR^d_+)$ for some $n\in\NN^d_0$ then one easily verifies that
$\sigma\in(\SSS(\RR^d_+))'$. Since the canonical inclusion
$G^{2\alpha}_{2\alpha}(\RR^d_+)\rightarrow \SSS(\RR^d_+)$ is
continuous and dense, $(\SSS(\RR^d_+))'$ is continuously injected
into $(G^{2\alpha}_{2\alpha}(\RR^d_+))'$, hence $\sigma\in
(G^{2\alpha}_{2\alpha}(\RR^d_+))'$.

\begin{lem}\label{jednakost duali}
Let $\alpha\geq 1/2$ and $\sigma_n$, $n\in\NN^d_0$, be measurable
functions on $\RR^d_+$ such that
$\sigma_n(\rho)/(\mathbf{1}+\rho)^{n/2}\in L^2(\RR^d_+)$, for all
$n\in\NN^d_0$ and for each $A>0$, \beas
\sum_{n\in\NN^d_0}\left\|\sigma_n(\rho)/(\mathbf{1}+\rho)^{n/2}\right\|_{L^2(\RR^d_+)}A^{|n|}n^{\alpha
n}<\infty. \eeas
Then $\sum_{n\in\NN^d_0} \sigma_n$ converges absolutely in $(G^{2\alpha}_{2\alpha}(\RR^d_+))'$.\\
\indent For each $n\in\NN^d_0$,
$\tilde{\sigma}_n(x,\xi)=\sigma_n(2v(x,\xi))$ is measurable on
$\RR^{2d}$ and $\tilde{\sigma}_n(x,\xi)/
(\mathbf{1}+2v(x,\xi))^{n/2}\in L^2(\RR^{2d})$. Moreover,
$\sum_{n\in\NN^d_0} \tilde{\sigma}_n(x,\xi)$ converges absolutely
in $(\SSS^{\alpha}_{\alpha}(\RR^{2d}))'$.
\end{lem}

\begin{proof} To prove that $\sum_{n\in\NN^d_0} \sigma_n$
converges absolutely in $(G^{2\alpha}_{2\alpha}(\RR^d_+))'$ let
$B$ be bounded subset of $G^{2\alpha}_{2\alpha}(\RR^d_+)$. For
each $f\in B$ denote by $a_{n,f}=\langle f,\mathcal{L}_n\rangle$.
By Theorem \ref{isoss}, $\{\{a_{n,f}\}_{n\in\NN^d_0}|\, f\in B\}$
is bounded in $\sss^{2\alpha}$ and hence also bounded in
$\sss^{2\alpha,a}$ for some $a>1$, i.e. there exists $C_0>0$ such
that $|a_{n,f}|\leq C_0 a^{-|n|^{1/(2\alpha)}}$ for all $f\in B$.
For $f\in B$, $n\in\NN^d_0$, we have \beas
\left|\langle\sigma_n,f\rangle\right|&\leq&
\sum_{k\in\NN^d_0}|a_{k,f}|
\int_{\RR^d_+}|\sigma_n(\rho)||\mathcal{L}_k(\rho)|d\rho\\
&\leq&
C_0\left\|\sigma_n(\rho)/(\mathbf{1}+\rho)^{n/2}\right\|_{L^2(\RR^d_+)}
\sum_{k\in\NN^d_0}a^{-|k|^{1/(2\alpha)}}\sum_{m\leq
n}\binom{n}{m}\|\rho^{m/2}\mathcal{L}_k\|_{L^2(\RR^d_+)}. \eeas As
in the first part of the proof of Proposition \ref{D Lemma3.1}, by
(\ref{estforlaged}), there exist $C_1,A>1$ which depend on $a$ but
not on $n\in\NN^d_0$ such that \beas
\sum_{k\in\NN^d_0}a^{-|k|^{1/(2\alpha)}}\sum_{m\leq
n}\binom{n}{m}\|\rho^{m/2}\mathcal{L}_k\|_{L^2(\RR^d_+)} \leq
C_1A^{|n|}n^{\alpha n}. \eeas Hence, by the assumption on
$\sigma_n$, $n\in\NN^d_0$, we have $\sum_{n\in\NN^d_0}\sup_{f\in
B} |\langle \sigma_n,f\rangle|<\infty$, i.e.
$\sum_{n\in\NN^d_0}\sigma(\rho)$
converges absolutely in $(G^{2\alpha}_{2\alpha}(\RR^d_+))'$.\\
\indent Next we prove that for each $n\in\NN^d_0$, $\tilde{\sigma}_n$ is measurable on $\RR^{2d}$. Firstly, we show the following:\\
\indent Let $v_1:\RR^{2d}\rightarrow \overline{\RR^d_+}$ be
defined by $v_1(x,\xi)=(2x_1^2+2\xi_1^2,\ldots,2x_d^2+2\xi_d^2)$.
If $g:\overline{\RR^d_+}\rightarrow \CC$ is measurable then $f:\RR^{2d}\rightarrow \CC$, $f=g\circ v_1$, is also measurable.\\
\indent For brevity in notation we denote by $\lambda_d$ and
$\lambda_{2d}$ the Lebesgue measure on $\RR^d$ and $\RR^{2d}$,
respectively. We prove that if $N\subseteq\overline{\RR^d_+}$ with
$\lambda_d(N)=0$ then $\lambda_{2d}(v_1^{-1}(N))=0$. Observe that
this implies the measurability of $f$ since every measurable set
is the union of a Borel set and a set of measure zero and the
preimage of every Borel set under $v_1$ is Borel set (since $v_1$
is continuous). Let $N\subseteq \overline{\RR^d_+}$ with
$\lambda_d(N)=0$. Denote by $N_1=N\cap\RR^d_+$ and
$N_2=N\backslash N_1$. Obviously
$\lambda_{2d}\left(v_1^{-1}(\overline{\RR^d_+}\backslash\RR^d_+)\right)=0$,
thus $v_1^{-1}(N_2)$ is measurable and has measure zero. It
remains to prove that $\lambda_{2d}(v_1^{-1}(N_1))=0$. Let
$\varepsilon>0$ be arbitrary but fixed. Since $\lambda_d(N_1)=0$,
there exists an open set $O\subseteq \RR^d_+$, such that
$N_1\subseteq O$ and $\lambda_d(O)<\varepsilon/\pi^d$. There exist
countable number of cubes $B(\rho^{(j)},r_j)=\{\rho\in\RR^d_+|\,
\rho^{(j)}_l\leq \rho_l<\rho^{(j)}_l+r_j,\, l=1,\ldots, d\}$,
$j\in\NN$, which are pairwise disjoint and $O=\bigcup_{j\in\NN}
B(\rho^{(j)},r_j)$ (cf. \cite[p. 50]{rudinn}). Observe that
$$\varepsilon/\pi^d>\lambda_d(O)=\sum_{j\in\NN}\lambda_d(B(\rho^{(j)},r_j))= \sum_{j\in\NN}r_j^d$$
and
$$v_1^{-1}(B(\rho^{(j)},r_j))=\prod_{l=1}^d \left\{(x_l,\xi_l)|\,
\rho^{(j)}_l/2\leq x_l^2+\xi_l^2<\rho^{(j)}_l/2+r_j/2\right\}.$$
\noindent Thus
$\lambda_{2d}(v_1^{-1}(B(\rho^{(j)},r_j))=r_j^d\pi^d/2^d$. Hence
$\lambda_{2d}(v_1^{-1}(O))=\sum_{j\in\NN}r_j^d\pi^d/2^d<\varepsilon$.
Since $\varepsilon>0$ is arbitrary, we conclude that
$v_1^{-1}(N_1)$ is measurable and it has measure zero.\\
\indent This fact readily implies the measurability of
$\tilde{\sigma}_n$. Moreover, \beas
\left\|\tilde{\sigma}_n(x,\xi)/(\mathbf{1}+2v(x,\xi))^{n/2}\right\|^2_{L^2(\RR^{2d})}=2^{-d}\pi^d
\left\|\sigma_n(\rho)/(\mathbf{1}+\rho)^{n/2}\right\|^2_{L^2(\RR^d_+)}.
\eeas Clearly, $\tilde{\sigma}_n\in
(\SSS^{\alpha}_{\alpha}(\RR^{2d}))'$ for each $n\in\NN^d_0$. To
prove that $\sum_{n\in\NN^d_0} \tilde{\sigma}_n$ converges
absolutely in $(\SSS^{\alpha}_{\alpha}(\RR^{2d}))'$, let $B$ be a
bounded subset of $\SSS^{\alpha}_{\alpha}(\RR^{2d})$. As the
latter space is the inductive limit of
$\ds\lim_{\substack{\longrightarrow\\ A\rightarrow \infty}}
\SSS^{\alpha,A}_{\alpha,A}(\RR^{2d})$ with compact linking
mappings, there exist $C,A\geq 1$ such that

$$\left\|x^n \xi^m D^p_x D^q_{\xi} f(x,\xi)\right\|_{L^2(\RR^{2d})}
\leq CA^{|n+m+p+q|}n!^{\alpha}m!^{\alpha}p!^{\alpha}q!^{\alpha},$$
$\forall n,m,p,q\in\NN^d_0$, $\forall f\in B$. For $f\in B$, we
have \beas \left|\langle\tilde{\sigma}_n,f\rangle\right|&\leq&
\left\|\tilde{\sigma}_n(x,\xi)/
(\mathbf{1}+2v(x,\xi))^{n/2}\right\|_{L^2(\RR^{2d})}
\left\|f(x,\xi)(\mathbf{1}+2v(x,\xi))^{n/2}\right\|_{L^2(\RR^{2d})}\\
&\leq& \pi^d
2^{|n|}\left\|\sigma_n(\rho)/(\mathbf{1}+\rho)^{n/2}\right\|^2_{L^2(\RR^d_+)}
\sum_{m+k+p=n}\frac{n!}{m!k!p!}\left\|x^m\xi^kf(x,\xi)\right\|_{L^2(\RR^{2d})}\\
&\leq& C\pi^d
(6A)^{|n|}n!^{\alpha}\left\|\sigma_n(\rho)/(\mathbf{1}+\rho)^{n/2}\right\|^2_{L^2(\RR^d_+)}.
\eeas Hence, by the assumption in the lemma, $\sum_{n\in\NN^d_0}
\sup_{f\in B}|\langle\tilde{\sigma}_n,f\rangle|< \infty$, i.e.
$\sum_{n\in\NN^d_0} \tilde{\sigma}_n$ absolutely converges in
$(\SSS^{\alpha}_{\alpha}(\RR^{2d}))'$. \qed
\end{proof}

Let $\sigma_n$ and $\tilde{\sigma}_n$, $n\in\NN^d_0$, be as in the
previous lemma and $\tilde{\sigma}(x,\xi)=\sum_{n\in\NN^d_0}
\tilde{\sigma}_n(x,\xi)\in (\SSS^{\alpha}_{\alpha}(\RR^d))'$. The
Weyl pseudo-differential operator $W_{\tilde{\sigma}}$ is a
continuous mapping from $\SSS^{\alpha}_{\alpha}(\RR^d)$ into
$(\SSS^{\alpha}_{\alpha}(\RR^d))'$. In this case, we obtain
improvement with the following result.

\begin{thm}
Let $\alpha\geq 1/2$. Let $\sigma_n(\rho)$ and
$\tilde{\sigma}_n(x,\xi)=\sigma_n(2v(x,\xi))$, $n\in\NN^d_0$, be
as in Lemma \ref{jednakost duali}. Then the Weyl
pseudo-differential operator $W_{\tilde{\sigma}}$ with a symbol
$\tilde{\sigma}(x,\xi)=\sum_{n\in\NN^d_0}
\tilde{\sigma}_n(x,\xi)\in (\SSS^{\alpha}_{\alpha}(\RR^{2d}))'$ is
a continuous mapping from
$\mathcal{S}^{\alpha}_{\alpha}(\mathbb{R}^d)$ into
$\mathcal{S}^{\alpha}_{\alpha}(\mathbb{R}^d)$ and it extends to a
continuous mapping from
$(\mathcal{S}^{\alpha}_{\alpha}(\mathbb{R}^d))'$ to
$(\mathcal{S}^{\alpha}_{\alpha}(\mathbb{R}^d))'$.\\
\indent Assume that for each $j\in \NN$,
$\sigma^{(j)}_n\in(G^{2\alpha}_{2\alpha}(\RR^d_+))'$,
$n\in\NN^d_0$, be as above and denote by
$\sigma^{(j)}=\sum_{n\in\NN^d_0}\sigma^{(j)}_n\in
(G^{2\alpha}_{2\alpha}(\RR^d_+))'$. If $\sigma^{(j)}\rightarrow
\sigma$ in $(G^{2\alpha}_{2\alpha}(\RR^d_+))'$ with $\sigma$ as
above, then $W_{\tilde{\sigma}^{(j)}}\rightarrow
W_{\tilde{\sigma}}$ in the strong topology of
$\mathcal{L}(\SSS^{\alpha}_{\alpha}(\RR^d),\SSS^{\alpha}_{\alpha}(\RR^d))$
and
$\mathcal{L}((\SSS^{\alpha}_{\alpha}(\RR^d))',(\SSS^{\alpha}_{\alpha}(\RR^d))')$.
\end{thm}

\begin{proof} Denote $\sigma=\sum_{n\in\NN^d_0}\sigma_n\in (G^{2\alpha}_{2\alpha}(\RR^d_+))'$ (cf. Lemma
\ref{jednakost duali}). Let
$f,g\in\mathcal{S}^{\alpha}_{\alpha}(\mathbb{R}^d)$. Denote
$f_k=\langle f,h_k\rangle$, $g_k=\langle g,h_k\rangle$ and
$s_{k}=(2\pi)^{d/2}(-1)^{|k|}2^{-d}\langle\sigma,\mathcal{L}_k\rangle$.
Similarly as in the first part of the proof of Theorem
\ref{ter11}, one obtains \beas (W_{\tilde{\sigma}}
f)(g)=(2\pi)^{-d/2}\sum_{(m,k)\in\mathbb{N}^{2d}_0}f_mg_k\langle\tilde{\sigma},\psi_{m,k}\rangle,
\eeas where $\psi_{m,k}=W(h_m,h_k)$ and the sum converges
absolutely. Next,
\begin{eqnarray*}
\langle\tilde{\sigma}(x,\xi),\psi_{m,k}(x,\xi)\rangle= \sum_{n\in
\NN^d_0} \int_{\RR^{2d}}
\sigma_n(2v(x,\xi))\psi_{m,k}(x,\xi)dxd\xi.
\end{eqnarray*}
By the same technique as in the proof of Theorem \ref{ter11},
$$\int_{\RR^{2d}} \sigma_n(2v(x,\xi))\psi_{m,k}(x,\xi)dxd\xi
=C_{n,m,k}\prod_{r=1}^d\int_{-\pi}^{\pi}
e^{-i(m_r-k_r)\theta_r}d\theta_r.$$
Thus $\langle\tilde{\sigma},\psi_{m,k}\rangle=0$ for $m\neq k$. Moreover,\\
\\
$\ds \int_{\RR^{2d}} \sigma_n(2v(x,\xi))\psi_{k,k}(x,\xi)dxd\xi$

\begin{eqnarray*}
&=&(2\pi)^{d/2}(-1)^{|k|}2^d\int_{\mathbb{R}^d_+}
\sigma_n(2\rho^2_1\ldots,2\rho^2_d)L_{k}(2\rho^2_1,\ldots,2\rho^2_d)e^{-|\rho|^2} \rho^{\mathbf{1}} d\rho\nonumber\\
&=&(2\pi)^{d/2}(-1)^{|k|}2^{-d}\langle
\sigma_n,\mathcal{L}_{k}\rangle.
\end{eqnarray*}

\noindent Thus,
$\langle\tilde{\sigma}(x,\xi),\psi_{k,k}(x,\xi)\rangle=(2\pi)^{d/2}(-1)^{|k|}2^{-d}
\langle \sigma,\mathcal{L}_k\rangle=s_k$. Hence, we obtain
\begin{equation*}\label{Wong OSNOVNA FORMULA d dim}
(W_{\tilde{\sigma}}
f)(g)=(2\pi)^{-d/2}\sum_{k\in\mathbb{N}^d_0}f_kg_ks_k
\end{equation*}
and the series converges absolutely since
$\{f_k\}_{k\in\mathbb{N}_0^d},\{g_k\}_{k\in\mathbb{N}_0^d}\in
\sss^{2\alpha}$ (see Proposition \ref{herexpult}) and
$\{s_k\}_{k\in\mathbb{N}_0^d}\in (\sss^{2\alpha})'$ (see Theorem
\ref{D ocena koef za dual}).
Observe that for each $n\in\NN^d_0$, $(W_{\tilde{\sigma}}
f)(h_n)=f_ns_n$. Since $\{s_n\}_{n\in\NN^d_0}\in
(\sss^{2\alpha})'$ and $\{f_n\}_{n\in\NN^d_0}\in \sss^{2\alpha}$,
we have $\{f_ns_n\}_{n\in\NN^d_0}\in \sss^{2\alpha}$, i.e.
$W_{\tilde{\sigma}} f\in\SSS^{\alpha}_{\alpha}(\RR^d)$ (by
Proposition \ref{herexpult}). We conclude that $f\mapsto
W_{\tilde{\sigma}} f$, $\SSS^{\alpha}_{\alpha}(\RR^d)\rightarrow
\SSS^{\alpha}_{\alpha}(\RR^d)$, is a well defined linear mapping.
Moreover, $W_{\tilde{\sigma}}f=\sum_{n\in\NN^d_0}f_n s_nh_n$. To
prove the continuity let $B$ be a bounded subset of
$\SSS^{\alpha}_{\alpha}(\RR^d)$. As $\{s_k\}_{k\in\NN^d_0}\in
(\sss^{2\alpha})'$, the set $\{\{f_ns_n\}_{n\in\NN^d_0}|\, f\in
B\}$ is bounded in $\sss^{2\alpha}$, thus
$\{W_{\tilde{\sigma}}f|\, f\in B\}$ is bounded in
$\SSS^{\alpha}_{\alpha}(\RR^d)$. As
$\SSS^{\alpha}_{\alpha}(\RR^d)$ is bornological, $f\mapsto
W_{\tilde{\sigma}}f$, $\SSS^{\alpha}_{\alpha}(\RR^d)\rightarrow
\SSS^{\alpha}_{\alpha}(\RR^d)$, is continuous. By similar
technique, one proves that for each $f\in
(\SSS^{\alpha}_{\alpha}(\RR^d))'$, $W_{\tilde{\sigma}}f
=\sum_{n\in\NN^d_0}f_ns_n h_n\in (\SSS^{\alpha}_{\alpha}(\RR^d))'$
and the mapping, $f\mapsto W_{\tilde{\sigma}}f$,
$(\SSS^{\alpha}_{\alpha}(\RR^d))'\rightarrow (\SSS^{\alpha}_{\alpha}(\RR^d))'$, is continuous.\\
\indent Let $\sigma,\sigma^{(j)}\in
(G^{2\alpha}_{2\alpha}(\RR^d_+))'$, $j\in \NN$, be as assumed in
the theorem, with $\sigma^{(j)}\rightarrow \sigma$ in
$(G^{2\alpha}_{2\alpha}(\RR^d_+))'$. In order to prove
$W_{\tilde{\sigma}^{(j)}}\rightarrow W_{\tilde{\sigma}}$ in the
strong topology of
$\mathcal{L}(\SSS^{\alpha}_{\alpha}(\RR^d),\SSS^{\alpha}_{\alpha}(\RR^d))$
(resp. in the strong topology of
$\mathcal{L}((\SSS^{\alpha}_{\alpha}(\RR^d))',(\SSS^{\alpha}_{\alpha}(\RR^d))')$)
it is enough to prove that for each
$f\in\SSS^{\alpha}_{\alpha}(\RR^d)$ (resp. for each
$f\in(\SSS^{\alpha}_{\alpha}(\RR^d))'$),
$W_{\tilde{\sigma}^{(j)}}f\rightarrow W_{\tilde{\sigma}}f$ in
$\SSS^{\alpha}_{\alpha}(\RR^d)$ (resp. in
$(\SSS^{\alpha}_{\alpha}(\RR^d))'$) since in this case the
Banach-Steinhaus theorem implies convergence in the topology of
precompact convergence and as $\SSS^{\alpha}_{\alpha}(\RR^d)$
(resp. $(\SSS^{\alpha}_{\alpha}(\RR^d))'$) is Montel the
convergence also holds in the strong topology. Thus for the fixed
$f\in\SSS^{\alpha}_{\alpha}(\RR^d)$ (resp. $f\in
(\SSS^{\alpha}_{\alpha}(\RR^d))'$). Theorem \ref{D ocena koef za
dual} implies that $\{s^{(j)}_k\}_{k\in\NN^d_0}\rightarrow
\{s_k\}_{k\in\NN^d_0}$ in $(\sss^{2\alpha})'$. But then
$\{f_ks^{(j)}_k\}_{k\in\NN^d_0}\rightarrow
\{f_ks_k\}_{k\in\NN^d_0}$ in $\sss^{2\alpha}$ (resp. in
$(\sss^{2\alpha})'$), i.e. $W_{\tilde{\sigma}^{(j)}}f\rightarrow
W_{\tilde{\sigma}}f$ in $\SSS^{\alpha}_{\alpha}(\RR^d)$ (resp. in
$(\SSS^{\alpha}_{\alpha}(\RR^d))'$). \qed
\end{proof}

By the similar arguments, one can proof the following theorem.

\begin{thm}
Let $\sigma$ be a measurable function on $\RR^d_+$ such that there
exists $n\in\NN^d_0$ for which $\sigma(\rho)/(\mathbf{1}+\rho)^n
\in L^2(\RR^d_+)$. Then
$\tilde{\sigma}(x,\xi)=\sigma(2v(x,\xi))\in (\SSS(\RR^{2d}))'$.
The Weyl pseudo-differential operator $W_{\tilde{\sigma}}$ is a
continuous mapping from $\mathcal{S}(\mathbb{R}^d)$ into
$\mathcal{S}(\mathbb{R}^d)$ and it extends to continuous mapping from $(\SSS(\RR^d))'$ into $(\SSS(\RR^d))'$.\\
\indent Let $\sigma^{(j)}$, $j\in\NN$, be measurable functions on
$\RR^d_+$ such that for each $j\in\NN$ there exists
$n^{(j)}\in\NN^d_0$ for which
$\sigma_j(\rho)/(\mathbf{1}+\rho)^{n^{(j)}}\in L^2(\RR^d_+)$. If
$\sigma$ is a measurable function on $\RR^d_+$ with the properties
stated above and if $\sigma^{(j)}\rightarrow \sigma$ in
$(\SSS(\RR^d_+))'$, then $W_{\tilde{\sigma}^{(j)}}\rightarrow
W_{\tilde{\sigma}}$ in the strong topology of
$\mathcal{L}(\SSS(\RR^d),\SSS(\RR^d))$ and
$\mathcal{L}((\SSS(\RR^d))',(\SSS(\RR^d))')$.
\end{thm}



\begin{thebibliography}{00}

\bibitem{abrsteg} M.~Abramowitz and I. A.~Stegun, Handbook of mathematical functions: with formulas, graphs, and mathematical tables,
No. 55. Courier Corporation, 1964.

\bibitem{A} A. Avantaggiati, S-spaces by means of the behaviour of
Hermite-Laguerre coefficients, Boll. Unione Mat. Ital., 4 (1985),
487-495.

\bibitem{Kaminski} R. D. Carmichael, A. Kami\.{n}ski, S. Pilipovic, Boundary Values and Convolution in
Ultradistribution Spaces, World Scientific, 2007.

\bibitem{D4}
A. J. Duran, Laguerre expansions of Tempered Distributions and
Generalized Functions, J. Math. Anal. Appl. 150 (1990), 166-180.

\bibitem{durlp} A.J.~Duran, A bound on the Laguerre polynomials, Stud. Math. 100 (2) (1991), 169-181.

\bibitem{D3}
A. J. Duran, The anlytic functionals in the lower half plane as a
Gel'fand-Shilov space, Math. Nachr. 153 (1991), 145-167.

\bibitem{D2}
A. J. Duran, Gel'fand-Shilov spaces for the Hankel transform,
Indag. Math. 3 (1992), 137-151.

\bibitem{D1}
A. J. Duran, Laguerre expansions of Gel'fand-Shilov spaces, J.
Approx. Theory 74 (1993), 280-300.

\bibitem{Ed}
A. Erdelyi, Higher Transcedentals Function, Vol. 2, McGraw-Hill,
New York, 1953.

\bibitem{GS1} I. M. Gel'fand, G. E. Shilov, Les Distributions,
Vol. 2, Dunod, Paris, 1965.

\bibitem{GPR} T. Gramchev, S. Pilipovic, L. Rodino, Classes of
Degenerate Elliptic Operators in Gelfan-Shilov Spaces,  Oper.
Theory Adv. Appl., Vol. 189, 15-31, Birkh\"{a}user, 2008.

\bibitem{GPR1} T. Gramchev, S. Pilipovic, L. Rodino, Global
regularity and stability in S-spaces for classes of degenerate
Shubin operators, Pseudo-differential operators: complex analysis
and partial differential equations, 81-90, Oper. Theory Adv.
Appl., 205, Birkh\"{a}user Verlag, Basel, 2010.

\bibitem{GPR2}T. Gramchev, A. Lecke, S. Pilipiovic, L. Rodino,
Gelfand-Shilov Type Spaces Through Hermite Expansions,
Pseudo-Differential Operators and Generalized Functions, Operator
Theory: Advances and Applications, Vol. 245, 2015.


\bibitem{Sm} S. Jaksic, B. Prangoski, Extension theorem of Whitney type for $\mathcal S(\mathbb{R}_+^d)$ by the use of the Kernel
Theorem, to appear in  Publ. Inst. Math. Beograd

\bibitem{kothe2} G. K\"{o}the, \textit{Topological vector spaces II}, Vol.II, Springer-Verlag, New York Inc., 1979.

\bibitem{lan} M.~Langenbruch,
Hermite functions and weighted spaces of generalized functions,
Manuscripta Mathematica 119(3) (2006), 269-285.

\bibitem{Le} Lebedev, Special functions and their applications,
Dover, New York, 1965.

\bibitem{M} B. S. Mitjagin, Nuclearity and other proporties of
spaces of type S, Amer. Math. Soc. Transl., Series 2 93 (1970),
45-59.

\bibitem{NR} F. Nicola, L. Rodino, Global Pseudo-Differential Calculus on Euclidean
Spaces, Pseudo-Differential Operators Volume 4, 2010.

\bibitem{pietsch} A.~Pietsch, Nuclear locally convex spaces, Springer-Verlag, Berlin-Heidelberg-New York, 1972.

\bibitem{SP} S. Pilipovic, Generalization of Zemanian spaces of
generalized functions which elements have series expansions,
 SIAM J. Math. Anal. 17 (1986), 477-484.

\bibitem{SP1} S.~Pilipovic, On the Laguerre expansions of generalized functions, C. R. Math. Rep. Acad. Sci. Canada 11 (1989),
no. 1, 23-27.

\bibitem{rudinn} W.~Rudin, Real and complex analysis, Tata McGraw-Hill Education, 1987.

\bibitem{Shubin} M. A. Shubin, Pseudodifferential Operators and Spectral
Theory, Springer, 2001.

\bibitem{teofanov} N.~Teofanov, Ultradistributions and time-frequency
analysis, Pseudo-differential operators and related topics,
Birkh\"{a}user Basel Vol. 164, (2006), 173-192.

\bibitem{Tr} F. Tr\' eves, Topological Vector Spaces,
Distributions and Kernels, Dover Publications, New York, 1995.

\bibitem{Wong} M. W. Wong, Weyl Transform, Springer-Verlag, New York, 1998.

\bibitem{Zayed} A. I. Zayed, Laguerre series as boundary values. SIAM J. Math. Anal. 13 (1982), no. 2, 263-279

\end{thebibliography}
\end{document}